\documentclass[12pt]{amsart}
\usepackage[english]{babel}
\usepackage[utf8]{inputenc}
\usepackage{amsmath, amsthm, amsfonts, amssymb}
\usepackage{mathtools}
\usepackage{xcolor}
\usepackage[colorlinks=true, linkcolor=red, citecolor=green, backref=page]{hyperref}
\usepackage{graphicx}
\usepackage[capitalise, noabbrev]{cleveref}
\usepackage{geometry}
\usepackage{enumitem}
\usepackage{marginnote}
\usepackage{tikz-cd}
\usepackage{centernot}
\usepackage{chngcntr}
\counterwithout{subsection}{section}

\newcommand{\overbar}[1]{\mkern 1.5mu\overline{\mkern-3mu#1\mkern-0.5mu}\mkern 0.5mu}
\newcommand{\dd}{\mathrm{d}}

\newcommand\restr[2]{{% we make the whole thing an ordinary symbol
  \left.\kern-\nulldelimiterspace % automatically resize the bar with \right
  #1 % the function
  \vphantom{\big|} % pretend it's a little taller at normal size
  \right|_{#2} % this is the delimiter
  }}

\DeclareMathOperator{\Ker}{Ker}

\DeclareMathOperator{\lcm}{lcm}

\DeclareMathOperator{\Id}{Id}
\DeclareMathOperator{\tr}{tr}

\newcommand{\jump}{\vskip 2mm}

\numberwithin{equation}{section}

\theoremstyle{plain}
\newtheorem{theorem}{Theorem}[section]
\newtheorem{proposition}[theorem]{Proposition}
\newtheorem{lemma}[theorem]{Lemma}
\newtheorem{corollary}[theorem]{Corollary}
\newtheorem*{conjecture}{Conjecture}

\theoremstyle{definition}
\newtheorem{definition}[theorem]{Definition}
\newtheorem{remark}{Remark}[section]

\numberwithin{theorem}{subsection}
\numberwithin{equation}{subsection}

\title[Yano's conjecture]{Yano's conjecture}

\author[G. Blanco]{Guillem Blanco}

\thanks{The author was supported by the grants Ministerio de Economía y Competitividad MTM2015-69135-P, Generalitat de Catalunya 2017SGR-932 and Agencia Estatal de Investigaci\'on PID2019-103849GB-I00. The author is supported by a Postdoctoral Fellowship of the Research Foundation -- Flanders}

\address{Department of Mathematics\\ KU Leuven,
Celestijnenlaan 200B, 3001 Leuven, Belgium.}
\email{guillem.blanco@kuleuven.be}

\begin{document}

\begin{abstract}
We present a proof of a conjecture proposed by Yano \cite{yano82} about the generic \( b \)-exponents of irreducible plane curve singularities.
\end{abstract}

\maketitle

\section{Introduction} \label{sec:introduction}

Let \( f : (\mathbb{C}^{n+1}, \boldsymbol{0}) \longrightarrow (\mathbb{C}, 0) \) be a germ of a holomorphic function. The \emph{local Bernstein-Sato} polynomial \( b_{f, \boldsymbol{0}}(s) \) of \( f \) is defined as the monic polynomial of smallest degree that fulfills the following functional equation
\[ P(s) \cdot f^{s+1} = b_{f, \boldsymbol{0}}(s) f^s,  \]
where \( P(s) \) is a differential operator in \( D_{\mathbb{C}^{n+1}, \boldsymbol{0}} \otimes \mathbb{C}[s] \), with \( D_{\mathbb{C}^{n+1}, \boldsymbol{0}} \) being the ring of local \( \mathbb{C} \)-linear differential operators and \( s \) a formal variable. The Bernstein-Sato polynomial \( b_{f}(s) \) was introduced in the algebraic case, independently, by Bernstein \cite{bernstein72} using the theory of algebraic \( D \)-modules and by Sato \cite{sato-shintani-90} in the context of prehomogeneous vector spaces. The existence of the Bernstein-Sato polynomial in the local case is due to Björk \cite{bjork74}. %The global Bernstein-Sato polynomial \( b_f(s) \) equals the least common multiple of all the local Bernstein-Sato polynomials \( b_{f, p}(s), p \in \mathbb{C}^{n+1} \), see \cite{narvaez91}.

\jump

If \( f : (\mathbb{C}^{n+1}, \boldsymbol{0}) \longrightarrow (\mathbb{C}, 0) \) defines an isolated singularity, there is a characterization of the Bernstein-Sato polynomial \( b_{f, \boldsymbol{0}}(s) \) in terms of the \emph{Gauss-Manin connection} of \( f \) given by Malgrange in \cite{malgrange75}. First, since \( -1 \) is always a root of \( b_{f, \boldsymbol{0}}(s) \), the \emph{reduced Bernstein-Sato polynomial} is defined as \( \tilde{b}_{f, \boldsymbol{0}}(s) := b_{f, \boldsymbol{0}}(s) / (s+1) \). Malgrange's result asserts that \( \tilde{b}_{f, \boldsymbol{0}}(s) \) is equal to the minimal polynomial of a certain action induced by the Gauss-Manin connection in a finite-dimensional complex vector space related to the \emph{Brieskorn lattice} of \( f \). In contrast, the \emph{\( b \)-exponents} are the roots of the characteristic polynomial of this action, see Sections \ref{sec:brieskorn-lattice} and \ref{sec:malgrange-theorem} for the exact definitions.

\jump

In general, very little is known about the roots of the Bernstein-Sato polynomial. They are negative rational numbers, see the works of Malgrange \cite{malgrange75} and Kashiwara \cite{kashiwara76}. A set of candidate roots can be obtained from the numerical data associated with a resolution of the singularity of \( f \), see Kashiwara \cite{kashiwara76} and Lichtin \cite{lichtin89}. However, this set of candidates is usually much larger than the actual set of roots of \( b_f(s) \). Also, the roots of \( b_{f}(-s) \) are bounded between zero and the dimension \( n + 1 \), see \cite{saito94}.

\jump

On the other hand, the roots of the Bernstein-Sato polynomial \( b_f(s) \) are connected to many different invariants associated with the singularity of \( f \). The roots of \( b_f(s) \) are related to the poles of the complex zeta function of \( f \), which was the original motivation of Bernstein \cite{bernstein72}. If \( \alpha \) is a root of \( b_{f, \boldsymbol{0}}(s) \), then \( \exp(-2 \pi \imath \alpha) \) is an eigenvalue of the local monodromy of \( f \), see the works \cite{malgrange75, malgrange83} of Malgrange. The \emph{spectral numbers} introduced by Steenbrink in \cite{steenbrink89} of an isolated singularity in the range \( (0, 1] \) are always the opposites in sign to some roots of the Bernstein-Sato polynomial, for the non-isolated singularity case we refer to \cite{budur03}. Similarly, the \emph{jumping numbers} associated with the \emph{multiplier ideals}, introduced in \cite{ELSV04}, which are in \( (0, 1] \), are roots of \( b_f(-s) \), see also \cite{budur-saito05} and \cite{lichtin86}.

\jump

However, the main difference between the roots of the Bernstein-Sato polynomial and most of the singularity invariants connected with \( b_{f}(s) \) is that the roots of \( b_f(s) \) are not topological invariants of the singularity. Precisely, this means that the roots of the Bernstein-Sato polynomial are not constant along the fibers of topologically trivial deformations of \( f \), see \cite{kato1, kato2} and \cite[\S 11]{varchenko80}, for some examples. It is then natural to wonder about the generic behavior of the roots of \( b_f(s) \), or equivalently of the \( b \)-exponents, along topologically trivial deformations of a singularity. This problem is still open even in the case of isolated plane curve singularities.

\jump

The goal of this paper is to contribute to this problem by proving a conjecture of Yano \cite{yano82} about the generic behavior of the \( b \)-exponents of irreducible plane curves along topologically trivial deformations of \( f \). Yano's conjecture states that given an irreducible plane curve singularity \( f \), that is, in the case \( n = 1 \), there exists some topologically trivial deformation of \( f \) such that, for generic fibers of the deformation, the \( b \)-exponents can be computed from the combinatorial data of the minimal resolution of singularities of \( f \). In addition, Yano proposes a set of candidates for these generic \( b \)-exponents in terms of the resolution data of \( f \). For the precise statement see \cref{sec:yanos-conjecture}.

\jump

Yano's conjecture was known to be true in the case that \( f \) has a single Puiseux pair, see the work of Cassou-Noguès \cite{cassou-nogues88}. More recently, Artal-Bartolo, Cassou-Noguès, Luengo, and Melle-Hernández \cite{ABCNLMH16}, proved the case of two Puiseux pairs under the hypothesis that the eigenvalues of the monodromy of \( f \) are pair-wise different. The author proved in \cite{blanco19} the conjecture for any number of Puiseux pairs also under the hypothesis that the eigenvalues of the monodromy are pair-wise different. This hypothesis ensures that the minimal polynomial giving \( b_{f, \boldsymbol{0}}(s) \) is equal to the characteristic polynomial giving the \( b \)-exponents. Then, under this hypothesis, the roots of the Bernstein-Sato polynomial of \( f \) can be determined from the poles of the complex zeta function of \( f \).

\jump

This work is organized as follows. The first part of the work contains results valid in any dimension. Sections \ref{sec:milnor-fiber} and \ref{sec:gauss-manin} introduce the usual terminology on the Milnor fiber, monodromy, and the Gauss-Manin connection. \Cref{sec:periods} introduces the periods of integrals along vanishing cycles on the Milnor fiber, following the work of Malgrange \cite{malgrange74, malgrange74bis}. In the next sections \ref{sec:geometric-sections} and \ref{sec:elementary-sections}, we introduce the geometric and elementary sections in the cohomology of the Milnor fiber from \cite{varchenko80}. Sections \ref{sec:brieskorn-lattice} and \ref{sec:malgrange-theorem}, present the relation between the Brieskorn lattice, the \( b \)-exponents, and the geometric sections using the results of Malgrange \cite{malgrange75} and Varchenko \cite{varchenko80}. In \cref{sec:semicontinuity}, we generalize a semicontinuity result of Varchenko \cite{varchenko80}. Following \cite{varchenko82}, \Cref{sec:resolution-sing} reviews the use of resolution of singularities and semi-stable reduction to study the periods of integrals. Finally, in \cref{sec:filtrations}, we construct the full asymptotic expansion of the periods in terms of a resolution of singularities.

\jump

The second part of this work deals with plane curve singularities. We present first, in \cref{sec:yanos-conjecture}, Yano's conjectures. After the work of Loeser \cite{loeser88}, \Cref{sec:multivalued} introduces cohomology with coefficients in a local system and a result of Deligne and Mostow \cite{deligne-mostow86}. \Cref{sec:plane-curves} shows which terms of the asymptotic expansions of the periods of integrals are non-zero using the structure of the minimal resolution of plane curve singularities. \cref{sec:dual-geometric} deals with the existence of certain dual locally constant geometric sections that are needed for the semicontinuity argument. \Cref{sec:char-poly} contains the proof of a technical result about the characteristic polynomial of the monodromy on some part of the Milnor fiber. In \cref{sec:plane-branches}, we introduce all the terminology related to the semigroups of irreducible plane curves, Teissier's monomial curve, and its deformations \cite{teissier-appendix}. Finally, we prove, in \cref{sec:generic-bexponents}, the semicontinuity of the \( b \)-exponents for plane branches and Yano's conjecture.

\jump

\textbf{Acknowledgments.} The author would like to thank his advisors, Maria Alberich-Carramiñana and Josep Àlvarez Montaner, for the fruitful discussions, the helpful comments and suggestions, and the constant support during the development of this work. The author would also like to thank Ben Lichtin for providing many helpful comments on a first draft of this work.

\section{General results}

\subsection{Milnor fiber} \label{sec:milnor-fiber}
Let \( f : ( \mathbb{C}^{n+1}, \boldsymbol{0}) \longrightarrow (\mathbb{C}, 0) \) be a germ of a holomorphic function defining an isolated singularity. For \( 0 < \delta \ll \epsilon \ll 1 \), let \( B_{\epsilon} \subset \mathbb{C}^{n+1} \) be the ball of radius \( \epsilon \) centered at the origin, \( T \subset \mathbb{C} \) the disk of radius \( \delta \) centered at zero, and \( T' \setminus \{0\} \) the punctured disk. Whenever is necessary to specify the radius of the disk we will write \( T_\delta \) and \( T'_{\delta} \). Abusing the notation, we will also denote by \( f \) a representative of the germ. Set
\begin{equation*}
X := B_{\epsilon} \cap f^{-1}(T), \quad X' := X \setminus f^{-1}(0), \quad X_t := B_{\epsilon} \cap f^{-1}(t),\quad t \in T.
\end{equation*}
The restriction \( f' : X' \longrightarrow T' \) is a smooth fiber bundle such that the diffeomorphism type of the fiber \( X_t \) is independent of the choice of \( \delta, \epsilon \) and \( t \in T' \), see \cite[\S 4]{milnor68}. Any of the fibers \( X_t \), or rather its diffeomorphism type, is called the \emph{Milnor fiber} of \( f \). The homology \( H_i (X_t, \mathbb{C}) \) (resp. cohomology \(H^i(X_t, \mathbb{C})\)) groups of the Milnor fiber are finite-dimensional vector spaces that vanish for \( i \geq n+1 \).

\jump

The action of the fundamental group \( \pi(T', t)  \) of the base of the fibration induces a diffeomorphism \( h \) of each fiber \( X_t \) which is usually called the \emph{geometric monodromy}. Alternatively, this same action induces \( h_* \) (resp. \( h^* \)) an endomorphism of \( H_*(X_t, \mathbb{C}) \) (resp. \( H^*(X_t, \mathbb{C}) \)), the \emph{algebraic monodromy} of the singularity \( f \). The fundamental result about the structure of the monodromy endomorphism is the so-called \emph{Monodromy Theorem}.
\begin{theorem}[Monodromy \cite{sga7-1, clemens69, brieskorn70}]
The operator \( h_* \) is quasi-unipotent, that is, there are integers \( p \) and \( q \) such that \( (h_*^p - \Id)^q = 0 \). In other terms, the eigenvalues of the monodromy are roots of unity. Moreover, one can take \( q = n + 1 \).
\end{theorem}
Since \( f \) defines an isolated singularity, the topology of the Milnor fiber is quite simple. In this case, the Milnor fiber has the homotopy type of a bouquet of \( \mu \) \( n \)-dimensional spheres, where \( \mu \) is the Milnor number of the singularity, see \cite[Thm. 6.5]{milnor68}. Therefore, \( \widetilde{H}_i(X_t, \mathbb{C}) = 0 \) for \( i \neq n \) and \( \dim_{\mathbb{C}} H_n(X_t, \mathbb{C}) =: \mu \). Furthermore, the Milnor number coincides with
\[ \mu = \dim_{\mathbb{C}} \frac{\Omega^{n+1}_{X, \boldsymbol{0}}}{\dd f \wedge \Omega^{n}_{X, \boldsymbol{0}}}. \]

\subsection{Gauss-Manin connection} \label{sec:gauss-manin}

The Milnor fibration \( f' : X' \longrightarrow T' \) of an isolated singularity \( f \) defines a holomorphic vector bundle \( f^* : H^n \longrightarrow T' \) on \( T' \), given by
\[ H^n := \bigcup_{t \in T'} H^n(X_t, \mathbb{C}), \]
the natural projection of \( H^n \) to \( T' \). Similarly, one has the dual vector bundle \( f_* : H_n \longrightarrow T' \). The vector bundle \( H^n \) (resp. \( H_n \)) is sometimes called the cohomological (resp. homological) Milnor fibration. Let us denote by \( \mathcal{H}^n \) (resp. \( \mathcal{H}_n \)) the locally free sheaf of holomorphic sections of the vector bundle \( H^n \) (resp. \( H_n\)). Since \( f' \) is a locally trivial fibration, \( H^n \) carries a natural integrable connection
\[ \nabla^* : \mathcal{H}^n \longrightarrow \Omega_{T'} \otimes_{\mathcal{O}_{T'}} \mathcal{H}^n \quad \textrm{or} \quad \nabla^*_{\partial/\partial_t} =: \partial^*_t : \mathcal{H}^n \longrightarrow \mathcal{H}^n, \]
called the \emph{Gauss-Manin connection} of the singularity. The kernel of the connection \( \Ker \nabla^* \) defines a \emph{local system}, which is a locally constant sheaf of complex vector spaces, whose sections are the locally constant sections of \( \mathcal{H}^n \) with respect to the connection. We will denote by \( \mathcal{L}^* := \Ker \nabla^* \) the sheaf of locally constant sections of \( \mathcal{H}^n \). In this way, one has that \( \mathcal{H}^n = \mathcal{L}^* \otimes_{\underline{\mathbb{C}}_{T'}} \mathcal{O}_{T'}\). In this setting, \( \mathcal{L}^* \) coincides with \( R^n f'_* \underline{\mathbb{C}}_{X'} \). The same constructions hold for the dual vector bundle \( H_n \). Brieskorn \cite[Satz 2]{brieskorn70} proved that the Gauss-Manin connection of an isolated singularity has \emph{regular singularities}, see also the work of Malgrange \cite[Thm. 4.1]{malgrange74}.

\jump

Denote by \( \Lambda \) the set of all eigenvalues of the monodromy operator \( h^* \). Fix \( \lambda \in \Lambda \), denote by \( H^n_{\lambda} \) the set of vectors of \( H^n \) that are annihilated by \( (h^* - \lambda \Id)^{n+1} \) and by \( f^*_{\lambda} \) the natural projection from \( H^n_{\lambda} \) to \( T' \). Then, \( H^n_{\lambda} \) is a holomorphic subbundle of \( H^n \) with a natural connection \( \nabla^*_{\lambda} \) coming naturally from \( \nabla^* \) and such that \( H^n = \bigoplus_{\lambda \in \Lambda} H^n_{\lambda} \). Let \( \mathcal{L}^*_{\lambda} = \ker \nabla^*_\lambda \) be the local system generated by the sections of \( \mathcal{L}^* \) that are annihilated by the endomorphism \( (h^* - \lambda \Id)^{n+1} \).

\subsection{Periods of integrals} \label{sec:periods}

Let \( \eta \in \Gamma(X, \Omega^n_X) \) be a holomorphic \( n \)-form on \( X \) and let \( \gamma(t), t \in T' \) be a locally constant section of the fibration \( H_n \). Since the restriction of \( \eta \) to \( X_t \) is a holomorphic form of maximal degree, and thus it is closed, the integral of \( \eta \) along the cycle \( \gamma(t) \)
\begin{equation*}
 I(t) := \int_{\gamma(t)} \eta
\end{equation*}
is well-defined. Furthermore, by the monodromy action, \( \gamma(t) \) is a multivalued function of \( t \) with values in \( H_n(X_t, \mathbb{C}) \) and \( I(t) \) defines a multivalued function on \( T' \). It is a holomorphic function since, by Leray's residue theorem, see for instance \cite[\S 1.5]{brieskorn70}, it holds that
\begin{equation} \label{eq:leray-residue}
I'(t) = \frac{d}{dt} \int_{\gamma(t)} \eta = \int_{\gamma(t)} \frac{\dd \eta}{\dd f},
\end{equation}
where \( \dd \eta / \dd f \) is the Gelfand-Leray form of \( \dd \eta \) which denotes the restriction to \( X_t \) of any holomorphic form \(\xi\) such that \( \xi \wedge \dd f = \dd \eta \). In local coordinates, if \( \dd \eta = g\, \dd x_0 \wedge \dots \wedge \dd x_n \), then, on the set \( \{ x \in X \ |\ \partial f/\partial x_i \neq 0 \} \), \( \dd \eta / \dd f \) is defined by the restriction to \( X_t \) of the form
\begin{equation} \label{eq:divide_df}
(-1)^{i+1}\bigg(\frac{\partial f}{\partial x_i}\bigg)^{-1} g \, \dd x_0 \wedge \cdots \wedge \widehat{\dd x_i} \wedge \cdots \wedge \dd x_n.
\end{equation}

Now, let \( \gamma(t) \) be a \emph{vanishing cycle}, that is, a locally constant section of \( H_n \) such that \( \gamma(t) \rightarrow 0 \) as \( t \rightarrow 0 \). Malgrange \cite[Lemma 4.5]{malgrange74} proves that if \( \gamma(t) \) is such a cycle, \( I(t) \rightarrow 0 \) as \( t \rightarrow 0 \) in any sector \( |\arg t\,| \leq C, C \in \mathbb{R}_{+} \). In addition, assume that \( \gamma(t) \) is a generalized eigenvector of the monodromy automorphism \( h_n \) of \( H_n(X_t, \mathbb{C}) \), i.e. \( (h_n - \lambda \Id )^p \gamma(t) = 0 \) with \( p \) minimal and \( \lambda \in \mathbb{C}^* \) a root of unity. Then, since the Gauss-Manin connection is regular, one has that \( I(t) \) has the following expansion, see \cite{malgrange74},
\begin{equation} \label{eq:integral-expansion-simple}
  I(t) = \int_{\gamma(t)} \eta = \sum_{\substack{\alpha \in L(\lambda)\\ 0 \leq q < p}} a_{\alpha, q} t^{\alpha}(\ln{t})^q, \quad a_{\alpha, q} \in \mathbb{C},
\end{equation}
where \( L(\lambda) \) is the set of \( \alpha > 0 \) such that \( \lambda = \exp({-2 \pi \imath \alpha}) \). Notice that if \( \gamma(t) \) is a generalized eigenvector of the monodromy but is not a vanishing cycle one can still consider the integral \( I(t) \). However, in this case, \( \lambda = 1 \) since \( I(t) \) has an expansion as in \cref{eq:integral-expansion-simple} plus a non-zero constant term.

\jump

Since for any top form \( \omega \in \Gamma(X, \Omega^{n+1}_X) \) there exists \(\eta \in \Gamma(X, \Omega^{n}_X)\) such that \( \omega = \dd \eta \), we will consider only the integrals of \( \dd \eta / \dd f \) along any vanishing cycle \( \gamma(t) \in H_n(X_t, \mathbb{C}) \). Then, by the Monodromy Theorem and \cref{eq:integral-expansion-simple}
\begin{equation} \label{eq:integral-expansion-total}
\int_{\gamma(t)} \frac{\omega}{\dd f} = \sum_{\lambda \in \Lambda} \sum_{\alpha \in L(\lambda)} \sum_{0 \leq k \leq n} \frac{1}{k!} a_{\alpha-1, k} t^{\alpha - 1} (\ln{t})^k.
\end{equation}

\subsection{Geometric sections} \label{sec:geometric-sections}

For every top form \( \omega \in \Gamma(X, \Omega^{n+1}_X) \) and every \( t \in T' \), the form \(\omega / \dd f\) in \( X_t \) defines an element of \( H^n(X_t, \mathbb{C}) \). Hence, every such \(\omega\) defines a section \( s[\omega] \) of the bundle \( H^n \). Indeed, by the previous section, if \( \gamma(t) \) is a locally constant section of the fibration \( H_n \), the pairing \( \langle s[ \omega ], \gamma \rangle \) given by \cref{eq:integral-expansion-total} is holomorphic and hence \( s[\omega] \) is a holomorphic section of the bundle \( H^n \), i.e. an element of \( \mathcal{H}^n \). Following the terminology of Varchenko \cite{varchenko80}, these sections will be called \emph{geometric sections}.

\jump

Given \( \varpi \) a local section of \( \mathcal{H}^n \) and \( \gamma \) a locally constant section of \( H_n \), one has that \( \frac{d}{dt} \langle \varpi, \gamma \rangle = \langle \partial^*_t \varpi, \gamma \rangle \), since horizontal sections generate \( \mathcal{H}^n \) over \( \mathcal{O}_{T'} \). Consequently, by \cref{eq:leray-residue}, the Gauss-Manin connection applied to the geometrical sections can be computed as \( \partial^*_t s[ \omega ] = s[ \dd( \omega / \dd f) ] \).

\jump

The complex numbers \( a_{\alpha, k} \) appearing in \cref{eq:integral-expansion-total} depend on \( \omega \) and \( \gamma(t) \). For a fixed \( \omega \), \( a_{\alpha, k} \) is a linear function on the space of locally constant sections of \( H_n \). As a consequence, the number \( a_{\alpha, k} \) defines a locally constant section \( A_{\alpha, k}^\omega(t) \) of the fibration \( H^n \) by the rule \( \langle A_{\alpha, k}^\omega(t), \gamma(t) \rangle := a_{\alpha, k}(\omega, \gamma) \). By construction, the geometrical sections \( s[ \omega ] \) are defined by
\begin{equation} \label{eq:geometric-section-definition}
s[\omega] := \sum_{\lambda \in \Lambda} \sum_{\alpha \in L(\lambda)} \sum_{0 \leq k \leq n} \frac{1}{k!} t^{\alpha-1} (\ln{t})^k A_{\alpha-1, k}^\omega(t).
\end{equation}
The sections \( A^{\omega}_{\alpha-1, k}(t) \) will be called \emph{locally constant geometric sections}.

\begin{lemma}[{\cite[Lemma 4]{varchenko80}}] \label{lemma:generators-local-system}
The local system \( \mathcal{L}^* \) is generated by the locally constant sections \( A_{\alpha-1, k}^\omega \), where \( \omega \in \Gamma(X, \Omega^{n+1}_X) \) and the \( \alpha, k \) are the same as in \cref{eq:integral-expansion-total}.
\end{lemma}

Let \( S_\alpha \) be the sheaf of all locally constant sections of the bundle \( H^n \) generated by the sections \( A_{\alpha, k}^\omega \) with fixed \( \alpha \in \mathbb{Q} \) and where \( \omega \in \Gamma(X, \Omega^{n+1}_X),\, k = 0, \dots, n \). Since \( A_{\alpha, k}^{\omega} = A_{\alpha+1, k}^{f \omega} \), one has that \( S_{\alpha} \subseteq S_{\alpha + 1} \). After \cref{lemma:generators-local-system}, \( \mathcal{L}^* = \sum_{\alpha} S_{\alpha} \).

\jump

Following \cite[\S 4]{varchenko80}, fix \( \lambda \in \Lambda \) and for every \( \alpha \in L(\lambda) -1 \) one defines the holomorphic subbundle \( f^*_{\lambda, \alpha} : H^n_{\lambda, \alpha} \longrightarrow T' \) of \( H^n_\lambda \) generated over \( \mathcal{O}_{T'} \) by the sections of \( S_{\alpha} \). According to \cite{varchenko80}, these subbundles have the following properties. If \( \alpha, \alpha' \in L(\lambda) -1 \) and \( \alpha > \alpha' \), then \( H^n_{\lambda, \alpha'} \) is a subbundle of \( H^n_{\lambda, \alpha} \) and for \( \alpha \in L(\lambda)-1 \) sufficiently large then \( H^n_{\lambda, \alpha} = H^n_{\lambda} \). Moreover, these subbundles are invariant under the covariant derivative of the connection \( \nabla^*_{\lambda} \) and the monodromy endomorphism \( h^* \).

\subsection{Elementary sections} \label{sec:elementary-sections}

We will proceed now to define the \emph{elementary sections} of \( \mathcal{H}^n \) in the sense of Varchenko \cite[\S 6]{varchenko80}. In order to do that, one needs to understand first the natural action of the monodromy on the local system \( \mathcal{L}^* \) in terms of the sections \( A_{\alpha, k}^{\omega}\). Let \( \omega \in \Gamma(X, \Omega^{n+1}_X), \lambda \in \Lambda, \alpha \in L(\lambda)-1 \) and assume that, at least, one of the sections \( A_{\alpha, 0}^\omega, \dots, A_{\alpha, n}^\omega \) is not equal to zero. Let \( p = \max \{ k \in \mathbb{Z} \ |\ A_{\alpha, k}^{\omega} \neq 0 \}  \).

\jump

The natural action of the monodromy \( h_* \) on the homology bundle \( H_n \) translates into a natural action of \( (h^*)^{-1} \) on the sections \( A_{\alpha, k}^\omega \) via the integral in \cref{eq:integral-expansion-total}. Namely, Varchenko shows in \cite[Lemma 5]{varchenko80} that
\begin{equation} \label{eq:monodromy-As}
  (h^*)^{-1} A_{\alpha, k}^\omega = \lambda^{-1} \sum_{j = k}^p \frac{(2 \pi \imath)^{j-k}}{(j-k)!} A_{\alpha, j}^\omega.
\end{equation}
Therefore, the sections \( A_{\alpha, 0}^\omega, \dots, A_{\alpha, p}^\omega \) are in the locally constant subsheaf \( \mathcal{L}^*_\lambda \) of \( H^n_{\lambda, \alpha} \) that is invariant under \( h^* \) and \( A_{\alpha, 0}^\omega \) is a cyclic section of this subsheaf. After \cref{eq:monodromy-As}, instead of the operator \( h^* - \lambda \Id \), we will consider the operator \((h^*)^{-1} \lambda - \Id \) on \( \mathcal{L}^*_{\lambda} \). Notice that, \(((h^*)^{-1} \lambda - \Id)^{n+1} = 0 \) on \( \mathcal{L}^*_{\lambda} \). Now, define the operator \( \ln{((h^*)^{-1}} \lambda) \) on \( \mathcal{L}^*_\lambda \) by the formula
\[ \ln{((h^*)^{-1} \lambda)} := \sum_{j = 1}^{\infty} \frac{(-1)^{j-1}}{j} ((h^*)^{-1} \lambda - \Id)^j \]
as in \cite[Lemma 5]{varchenko80}. Hence, from \cref{eq:monodromy-As},
\begin{equation} \label{eq:cyclic-generator}
  A_{\alpha, k}^{\omega} = \left( \frac{\ln{((h^*)^{-1} \lambda)}}{2 \pi \imath} \right)^k A_{\alpha, 0}^\omega,
\end{equation}
and, therefore,
\[ t^\alpha \sum_{k = 0}^p \frac{1}{k!} (\ln{t})^k A_{\alpha, k}^\omega(t) = \exp{\left[ \ln{t} \left( \alpha \Id + \frac{\ln{((h^*)^{-1} \lambda)}}{2 \pi \imath} \right) \right]} A_{\alpha, 0}^\omega(t). \]
Then, Varchenko \cite{varchenko80} defines the \emph{elementary section} associated with a locally constant section \( A \) of \(\mathcal{L}^*_{\lambda} \) and \( \alpha \in L(\lambda)-1 \) as the section \( s_\alpha[A] \) of \( \mathcal{H}_{\lambda} := \mathcal{L}^*_{\lambda} \otimes_{\underline{\mathbb{C}}_{T'}} \mathcal{O}_{T'} \) defined by
\begin{equation}
s_\alpha[A](t) :=   \exp{\left[ \ln{t} \left( \alpha \Id + \frac{\ln{((h^*)^{-1} \lambda)}}{2 \pi \imath} \right) \right]} A(t).
\end{equation}
We end this section with the following properties of elementary sections.
\begin{lemma}[{\cite[Lemma 9]{varchenko80}}] \label{lemma:elementary-sections}
Let \( \lambda \in \Lambda, \alpha \in L(\lambda)-1 \) and \( A \) a section of \( \mathcal{L}^*_\lambda \), then:
\begin{enumerate}
\item The sections \( s_\alpha[A] \) are holomorphic univalued sections of the vector bundle \( H^n_\lambda \).
\item If the sections \( A_0, \dots, A_p \) of \(\mathcal{L}_\lambda^* \) are linearly independent at every fiber, then the sections \( s_\alpha[A_0], \dots, s_\alpha[A_p] \) are linearly independent at every point \( t \in T' \).
\item The action of the covariant derivative on an elementary section is:
\[ t \partial^*_t s_\alpha[A] = \alpha s_\alpha[A] + (2 \pi \imath)^{-1}s_\alpha[ \ln{((h^*)^{-1} \lambda)}A]. \]
\end{enumerate}
\end{lemma}

\subsection{Brieskorn lattice} \label{sec:brieskorn-lattice}

In \cite{brieskorn70}, Brieskorn introduced the theory of local Gauss-Manin connections to give an algebraic description of the local monodromy of an isolated hypersurface singularity. In order to make an explicit computation of the monodromy, he considered the following \( \mathcal{O}_{T, 0} \)-module
\[ H''_{f, \boldsymbol{0}} := \frac{\Omega^{n+1}_{X, \boldsymbol{0}}}{\dd f \wedge \dd \Omega ^{n-1}_{X, \boldsymbol{0}}}, \]
which is called the \emph{Brieskorn lattice}. The rank of this module is equal to the Milnor number \( \mu \) of \( f \) and it is free of torsion by a result of Sebastiani \cite{sebastiani70}, see also \cite[Thm. 5.1]{malgrange74}. Let \( k \) be the field of fractions of \( \mathcal{O}_{T, 0} \) and consider the \( k \)-vector space \( H''_{f, \boldsymbol{0}} \otimes_{\mathcal{O}_{T,0}} k \). Then, the Brieskorn lattice \( {H}''_{f, \boldsymbol{0}} \) carries a meromorphic connection \( \partial_t : {H}''_{f, \boldsymbol{0}} \otimes k \longrightarrow {H}''_{f, \boldsymbol{0}} \otimes k \) defined by
\[ \partial_t \omega := f^{-\kappa_f} \psi - \kappa_f f^{-1} \omega, \qquad \textrm{with} \qquad \dd(f^{\kappa_f} \omega ) = \dd f \wedge \psi, \]
where \( \kappa_f \) is the minimum positive integer such that \( f^{\kappa_f} \in (\partial f/ \partial x_0, \dots, \partial f/ \partial x_n) \mathcal{O}_{X, \boldsymbol{0}} \). The connection \( \partial_t \) on the Brieskorn lattice is identified with the Gauss-Manin connection \( \partial^*_t \) of the singularity \( f \). The monodromy of the connection \( \partial_t \) coincides with the local monodromy of the isolated singularity \( f \) \cite[Satz 1]{brieskorn70}.

\jump

The relation between the Brieskorn lattice \( H''_{f, \boldsymbol{0}} \) and the geometrical sections of the cohomological bundle \( H^n \) is as follows. Let \( j : T' \xhookrightarrow{\quad} T \) be the open inclusion. Consider \( (j_*\mathcal{H}^n)_{0} \) the stalk at zero of the direct image by \( j \) of the sheaf \( \mathcal{H}^n \). By the results of Brieskorn \cite{brieskorn70}, \( H''_{f, \boldsymbol{0}} \) is identified with a \( \mathcal{O}_{T, 0} \)-submodule of \((j_* \mathcal{H}^n)_0 \) which generates \((j_* \mathcal{H}^n)_0 \) as a \((j_* \mathcal{O}_{T'})_0 \)-module. Indeed, if \( \omega \in \Omega^{n+1}_{X, \boldsymbol{0}} \) is the germ of a form representing \( \widetilde{\omega} \in H''_{f, \boldsymbol{0}} \), then \( s[\widetilde{\omega}](t) \) defines a cohomology class of \( H^{n}(X_t, \mathbb{C}) \) given by \( [\restr{\omega / \dd f}{X_t}] \), which does not depend on the representative. Furthermore, \( \partial^*_t s[\widetilde{\omega}] = s[\partial_t \widetilde{\omega}] \).

\subsection{Malgrange's theorem} \label{sec:malgrange-theorem}

Let \( b_{f,\boldsymbol{0}}(s) \in \mathbb{C}[s] \) be the local \emph{Bernstein-Sato polynomial} of \( f \). By definition \( b_{f, \boldsymbol{0}}(s) \) is the monic polynomial of smallest degree that satisfies the functional equation
\[ P(s) \cdot f^{s+1} = b_{f, \boldsymbol{0}}(s) f^s, \]
where \( P(s) \in D_{X, \boldsymbol{0}} \otimes_{\mathbb{C}} \mathbb{C}[s] \), with \( D_{X, \boldsymbol{0}} \) being the ring of local \( \mathbb{C} \)-linear differential operators and \( s \) a formal variable. In order to define the \( b \)-exponents, we are interested in the following characterization of the Bernstein-Sato polynomial of an isolated singularity in terms of the Brieskorn lattice given by Malgrange in \cite{malgrange75}. Since \( -1 \) is always a root of \( b_{f, \boldsymbol{0}}(s) \), one calls \( \tilde{b}_{f, \boldsymbol{0}}(s) := b_{f, \boldsymbol{0}}(s)/(s+1) \) the \emph{reduced Bernstein-Sato polynomial}.

\jump

Let \( \widetilde{H}''_{f, \boldsymbol{0}} := \sum_{k \geq 0}({t \partial_t})^k H''_{f, \boldsymbol{0}} \) be the saturation of the Brieskorn lattice with respect to the differential operator \( t \partial_t \). Here, the summation symbol means addition of \( \mathcal{O}_{T, 0} \)-modules and this sum is known to stabilize after the first \( \mu \) terms because \( \partial_t \) is regular. In this way, \( \widetilde{H}''_{f, \boldsymbol{0}} \) is a \( t \partial_t \)-stable lattice, that is, the connection \( \partial_t \) has a simple pole in \( \widetilde{H}''_{f, \boldsymbol{0}} \) and \( \widetilde{H}''_{f, \boldsymbol{0}} \) is isomorphic to \( H''_{f, \boldsymbol{0}} \) as \( k \)-vector spaces.

\begin{theorem}[{\cite[Thm. 5.4]{malgrange75}}]
The reduced Bernstein-Sato polynomial \( \tilde{b}_{f, \boldsymbol{0}}(s) \) equals the minimal polynomial of the endomorphism \( - \overline{\partial_t t} : \widetilde{H}''_{f, \boldsymbol{0}}/t \widetilde{H}''_{f, \boldsymbol{0}} \longrightarrow \widetilde{H}''_{f, \boldsymbol{0}}/t  \widetilde{H}''_{f, \boldsymbol{0}} \) of complex vector spaces of dimension \( \mu \).
\end{theorem}

The roots of the characteristic polynomial of the endomorphism
\[ \overline{\partial_t t} : \widetilde{H}''_{f, \boldsymbol{0}}/t \widetilde{H}''_{f, \boldsymbol{0}} \longrightarrow \widetilde{H}''_{f, \boldsymbol{0}}/t  \widetilde{H}''_{f, \boldsymbol{0}} \]
are called the \( b \)-\emph{exponents} of the isolated singularity \( f \). The sign change between the roots of \( b_{f, \boldsymbol{0}}(s) \) and the \( b \)-exponents is just a convention. Yano's conjecture is expressed in terms of the \( b \)-exponents of an isolated singularity \( f \), see \cref{sec:yanos-conjecture}.

\jump

Next, we will present the relation, given by Varchenko in \cite[\S 8]{varchenko80}, between the elementary sections introduced in \cref{sec:elementary-sections} and the saturation of the Brieskorn lattice.

\jump

Consider the quotient vector bundle \( H^n_{\lambda, \alpha} / H^n_{\lambda, \alpha -1} \) over \( T' \) which will be denoted by \( F_\alpha \). Let \( \mathcal{F}_\alpha \) be the locally free sheaf of sections of the vector bundle \( F_\alpha \). We will denote by \( \mathcal{G}_\alpha \) the subsheaf of \( \mathcal{F}_\alpha \) generated over \( \underline{\mathbb{C}}_{T'} \) by the image of the elementary sections \( s_\alpha[A] \), with \( A \) a section of \( S_\alpha \), under the projection map
\[ \pi_\alpha : H^n_{\lambda, \alpha} \longrightarrow H^n_{\lambda, \alpha} / H^n_{\lambda, \alpha-1} = F_\alpha. \]
After \cref{lemma:elementary-sections}, for every value \( t \in T' \), the sections \( \mathcal{G}_\alpha \) generate the whole fiber of \( F_\alpha \). The restriction of the connection \( \nabla^*_\lambda \) of \( H_\lambda^n \) to \( H^n_{\lambda, \alpha} \) induces a connection %\( \nabla^*_{\lambda, \alpha} \)%
in the quotient bundle \( F_\alpha \). %Therefore, since the sheaf \( \mathcal{G}_\alpha \) is annihilated by this connection \( \nabla^*_{\lambda, \alpha} \), \( \mathcal{G}_\alpha \) is a local system equal to \( \ker \nabla^*_{\lambda, \alpha} \). %
Furthermore, by \cref{lemma:elementary-sections}, the operator \(  t \partial^*_t \) maps elementary sections to elementary sections and thus, it induces an endomorphism \( D_\alpha \) on \( {\mathcal{G}}_\alpha \) which has eigenvalues equal to \( \alpha \) at every fiber. For more details see \cite[Lemma 10]{varchenko80}.

\jump

If \( j : T' \xhookrightarrow{\quad} T \) denotes again the open inclusion and \( j_{!} \) is the extension by zero, then we have that \( j_{!} \mathcal{G}_\alpha \neq j_* \mathcal{G}_\alpha \), meaning that the stalk \( (j_* \mathcal{G}_\alpha)_0 \) is not zero. Indeed, by \cref{lemma:elementary-sections}, the elementary sections whose image under \( \pi_\alpha \) generate \( \mathcal{G}_\alpha \) are univalued. We will continue to denote \( D_\alpha \) the extension to \( j_* \mathcal{G}_{\alpha} \) of the endomorphism \( D_\alpha \) of \( \mathcal{G}_{\alpha} \).

\begin{theorem}{\cite[Thm.~13]{varchenko80}} \label{thm:varchenko}
Let \( \mathcal{G}_{\alpha-1}, \alpha \in L(\lambda), \lambda \in \Lambda \) be the locally constant sheaves defined above and consider the locally constant sheaf \( {\mathcal{G}} := \bigoplus_{\lambda \in \Lambda} \bigoplus_{\alpha \in L(\lambda)} {\mathcal{G}}_{\alpha-1} \) of complex vector spaces with the endomorphism \( \overbar{D} := \bigoplus_{\lambda \in \Lambda} \bigoplus_{\alpha \in L(\lambda)} D_{\alpha-1} \). Then, there exists a natural isomorphism of complex vectors spaces between \( (j_*{\mathcal{G}})_0 \) and \( \widetilde{H}''_{f, \boldsymbol{0}} / t \widetilde{H}''_{f, \boldsymbol{0}} \) and under this isomorphism, the endomorphism \( \overline{t\partial_t} \) in \( \widetilde{H}''_{f, \boldsymbol{0}} / t \widetilde{H}''_{f, \boldsymbol{0}} \) corresponds to \( \overbar{D}_0 \).
\end{theorem}

The set of \( b \)-exponents of an isolated singularity is, therefore, after \cref{thm:varchenko}, contained in the set of positive rational numbers of the form \( \alpha \in L(\lambda), \lambda \in \Lambda \).

\subsection{Semicontinuity of the \texorpdfstring{\( b \)}{b}-exponents} \label{sec:semicontinuity}

In \cite{varchenko80}, Varchenko proves the semicontinuity of the \( b \)-exponents under \( \mu \)-constant deformations of the singularity in the case that the eigenvalues of the monodromy endomorphism are pair-wise different. In this section, we will generalize his result to any isolated singularity under the extra assumption of the existence of certain \emph{dual locally constant geometric sections}. First, let us review the results from \cite[\S 11]{varchenko80}.

\jump

Fix \( \lambda \in \Lambda \) an eigenvalue of the monodromy. We have seen in \cref{sec:geometric-sections} that the vector bundles \( H^n_{\lambda, \alpha}, \alpha \in L(\lambda)-1 \) form an increasing filtration in \( H^n_{\lambda} \). Denote by \( d_{\alpha} \) the dimension of the bundle \( H^n_{\lambda, \alpha} \). Then, \( d_{\alpha} \leq d_{\alpha +1} \) and \( d_{\alpha} = \dim_{\mathbb{C}} H^n_{\lambda} \), for \( \alpha \gg 0 \), which is exactly the number of eigenvalues of the monodromy that are equal to \( \lambda \). Then, for the quotient bundles \( F_\alpha \) defined in the previous section, we have that
\[ \sum_{\lambda \in \Lambda} \sum_{\alpha \in L(\lambda)} \dim_{\mathbb{C}} F_{\alpha-1} = \mu, \]
since \( \dim_{\mathbb{C}} F_{\alpha} = d_{\alpha} - d_{\alpha - 1} \). If one assumes that the monodromy has pair-wise different eigenvalues, then \( \dim_{\mathbb{C}} H^n_{\lambda} = 1 \) for all \( \lambda \in \Lambda \). Therefore, there is only a single \( \alpha \in L(\lambda)-1 \) that can be a \( b \)-exponent, and such \( \alpha \) is characterized by the fact that \( \dim_{\mathbb{C}} F_\alpha = 1 \).

\jump

Let \( f_y : (\mathbb{C}^{n+1}, \boldsymbol{0}) \longrightarrow (\mathbb{C}, 0) \), with \( y \in I_\eta := \{z \in \mathbb{C} \ |\ |z| < \eta\}, 0 < \eta \ll 1 \), be a one-parameter \( \mu \)-constant deformation of the isolated singularity \( f =: f_0 \). Recall that along the fibers of a \( \mu \)-constant deformation both the eigenvalues and the Jordan form of the monodromy endomorphism remain constant, \cite{trang-ramanujam76}. Then, if one denotes by \( d_\alpha(y) \) the dimension of the corresponding bundle \( H^n_{\lambda, \alpha}(y) \) of the isolated singularity \( f_y \), we have

\begin{proposition}[{\cite[Cor. 19]{varchenko80}}] \label{prop:semicontinuity-dimension}
The dimension \( d_\alpha(y) \) of \( H^n_{\lambda, \alpha}(y) \) depends lower-semicon\-tinuously on the parameter \( y \).
\end{proposition}

From this, and under the assumption that the eigenvalues of the monodromy are pair-wise different, Varchenko \cite[Cor. 21]{varchenko80} deduces a lower-semicontinuity for the roots of the Bernstein-Sato polynomial of \( f \). Since the \( b \)-exponents are the opposites in sign to the roots of \( b_{f, \boldsymbol{0}}(s) \), one has an analogous upper-semicontinuity for the \( b \)-exponents of \( f \). Next, we will construct suitable subbundles of \( H^n_{\lambda, \alpha} \) such that their dimension completely characterizes the existence of \( b \)-exponents even if the eigenvalues of the monodromy are not pair-wise different.

\jump

Let us fix \( \gamma_1(t), \dots, \gamma_\mu(t) \) locally constant cycles forming a basis of generalized eigenvectors of the monodromy endomorphism in the homology of each fiber \( X_t, t \in T' \). Similarly to the sheaf of locally constant sections \( S_\alpha, \alpha \in L(\lambda)-1 \) introduced in \cref{sec:geometric-sections}, define \( S_{\gamma_i}, i = 1, \dots, \mu \) the sheaf of locally constant sections of the bundle \( H^n \) generated by the locally constant geometric sections \( A^{\omega}_{\alpha, k},\, k = 0, \dots, n \) with \( \omega \in \Gamma(X, \Omega^{n+1}_X) \) such that
\begin{equation} \label{eq:duality}
  \langle A^{\omega}_{\alpha, k}(t), \gamma_i(t) \rangle \neq 0 \qquad \textrm{and} \qquad \langle A^{\omega}_{\alpha, k}(t), \gamma_j(t) \rangle = 0,
\end{equation}
with \( \gamma_j(t) \) any other eigenvector of the basis different from \( \gamma_i(t) \). The locally constant sections in \( S_{\gamma_{i}} \) will be called \emph{dual locally constant geometric sections} to \( \gamma_{i}(t) \). Notice that, after \cref{eq:integral-expansion-simple}, if \( \gamma(t) \) is a generalized eigenvector of eigenvalue \( \lambda \), then it is necessary that \( \lambda = \exp{(-2 \pi \imath \alpha)} \) for \( \langle A^\omega_{\alpha, k}(t), \gamma(t) \rangle \) to be non-zero.

\jump

It is not a priori clear that dual locally constant geometric sections should exist. In Sections \ref{sec:dual-geometric} and \ref{sec:generic-bexponents}, we will show that, for irreducible plane curve singularities, dual locally constant geometric sections exist with respect to a certain basis of eigenvectors of the monodromy. As we will see in Sections \ref{sec:dual-geometric} and \ref{sec:generic-bexponents}, the geometric picture behind the duality \eqref{eq:duality} is that two locally constant sections of \( H^n \) with the same \( \alpha \in L(\lambda)-1 \) will be linearly independent because they will be dual to two linearly independent eigenvectors of eigenvalue \( \lambda \). At the same time, these eigenvectors will be linearly independent because they will vanish to different rupture divisors of the minimal resolution of \( f \).

\jump

Let \( \gamma_{\lambda}(t) \) be a generalized eigenvector of the basis \( \gamma_{1}(t), \dots, \gamma_{\mu}(t) \) with eigenvalue \( \lambda \in \Lambda \). We define the holomorphic subbundle \( f^*_{\gamma_\lambda} : H^n_{\gamma_\lambda} \longrightarrow T' \) of the bundle \( H^n_{\lambda} \) as the bundle generated by the locally constant sections in \( S_{\gamma_\lambda} \). These subbundles are also invariant under the covariant derivative associated with the connection \( \nabla^*_{\lambda} \) in \( H^n_{\lambda} \). Consider the vector bundle \( H^n_{\gamma_{\lambda}} \cap H^n_{\lambda, \alpha} \), which will be denoted by \( H^n_{\gamma_{\lambda}, \alpha} \), and which is a subbundle of \( H^n_{\lambda, \alpha} \subset H_{\lambda}^n \). The bundle \( H^n_{\gamma_{\lambda}, \alpha} \) is also invariant by the covariant derivative of the connection \( \nabla^*_{\lambda} \) of \( H^n_{\lambda} \). Notice that \( \dim_{\mathbb{C}} H^n_{\gamma_\lambda} \leq 1 \) and if the monodromy has only one eigenvalue equal to \( \lambda \), then \( H^n_{\lambda} = H^n_{\gamma_\lambda} \).

\jump

Following \cref{sec:brieskorn-lattice}, define the quotient bundles \( H^n_{\gamma_{\lambda}, \alpha} / H^n_{\gamma_{\lambda}, \alpha - 1} \) which will be denoted by \( F_{\gamma_\lambda, \alpha} \). Let \( \mathcal{F}_{\gamma_{\lambda}, \alpha} \) denote the locally free sheaf of sections of \( F_{\gamma_\lambda, \alpha} \), then \( \mathcal{G}_{\gamma_\lambda, \alpha} \) is the subsheaf of \( \mathcal{F}_{\gamma_\lambda, \alpha} \) generated by the image of the elementary sections \( s_\alpha[A] \), with \( A \) a section in \( S_\alpha \cap S_{\gamma_{\lambda}} \), under the projection map
\[ \pi_{\gamma_\lambda, \alpha} : H^n_{\gamma_\lambda, \alpha} \longrightarrow H^n_{\gamma_\lambda, \alpha} / H^n_{\gamma_\lambda, \alpha-1} =: F_{\gamma_{\lambda}, \alpha}. \]
All the quotient bundles and subbundles of \( H^n \) presented so far are related by the following diagram,
\begin{equation*}
  \begin{tikzcd}
    H^n_{\gamma_{\lambda}, \alpha} \arrow[two heads]{d} \arrow[hook]{r} & H^n_{\lambda, \alpha} \arrow[two heads]{d} \arrow[hook]{r} & H^n_{\lambda} \arrow[hook]{r} & H^n  \\
    {F}_{\gamma_{\lambda}, \alpha} \arrow[hook]{r}  & {F}_\alpha \\
    {G}_{\gamma_{\lambda}, \alpha} \arrow[hook]{u} \arrow[hook]{r} & {G}_{\alpha}. \arrow[hook]{u}
  \end{tikzcd}
\end{equation*}
One can check that the subsheaf \( \mathcal{G}_{\gamma_\lambda, \alpha} \) has the same properties as the subsheaf \( \mathcal{G}_\alpha \) described in the previous section.

\jump

Assuming the existence of dual locally constant geometric sections for \( \gamma_{\lambda}(t) \), the important point is that the dimensions of the vector bundles \( H^n_{\gamma_\lambda, \alpha} \subseteq H^n_{\gamma_\lambda}, \alpha \in L(\lambda)-1 \), denoted \( d_{\gamma_\lambda, \alpha} \), are either zero or one. In addition, \( d_{\gamma_\lambda, \alpha} = \dim_{\mathbb{C}} H^n_{\gamma_\lambda} = 1, \) for \( \alpha \gg 0 \). Therefore, this construction allows us to characterize the existence of a certain \( b \)-exponent in terms of the dimensions of \( F_{\gamma_{\lambda}, \alpha} \). A candidate \( b \)-exponent \( \alpha \) for \( \alpha \in L(\lambda) \), associated with the generalized eigenvector \( \gamma_{\lambda} \) of the monodromy, is a \( b \)-exponent, if and only if, \( \dim_{\mathbb{C}} F_{\gamma_{\lambda}, \alpha-1} = d_{\gamma_{\lambda}, \alpha-1} - d_{\gamma_{\lambda}, \alpha - 2} = 1 \).

\jump

As before, let \( f_y :(\mathbb{C}^{n+1}, \boldsymbol{0}) \longrightarrow(\mathbb{C}, 0), y \in I_\eta \) be a one-parameter \( \mu \)-constant deformation of an isolated singularity. Following \cite[\S 11]{varchenko80} and the notations from \cref{sec:introduction}, let \( \mathcal{X} := \{ (x, y) \in B_\epsilon \times I_\eta \ |\ f_y(x) = t \in T_\delta\} \). Denote by \( \Phi : \mathcal{X} \longrightarrow T_\delta \times I_\eta \) the application given by \( (x, y) \mapsto (f_y(x), y) \). Let
\[ \mathcal{X}_{t, y} := \mathcal{X} \cap \Phi^{-1}(t, y) \qquad \textrm{and} \qquad \mathcal{X}' := \mathcal{X} \setminus \Phi^{-1}(\{0\} \times I_\eta), \]
then \( \Phi' : \mathcal{X}' \longrightarrow T'_\delta \times I_\eta \) is a locally trivial fibration with fibers \( \mathcal{X}_{t, y} \) that is independent of \( \epsilon, \delta, \eta \) if \( 0 < \delta, \eta \ll \epsilon \ll 1 \). This fibration is fiber-homotopic to the Milnor fibration of \( f_y \) for fixed \( y \in I_\eta \), see \cite{trang-ramanujam76}. As in \cref{sec:introduction}, this means that the associated (co)homological bundle carry an integrable connection. Furthermore, the (co)homological bundle of the restriction of \( \Phi' \) over \( T'_\delta \times \{y\} \) is canonically isomorphic to the (co)homological Milnor fibration of the singularity of the fiber \( f_y \), see \cite[Cor. 17]{varchenko80}.

\jump

In particular, this means that we can fix a basis \( \gamma_1(t, y), \gamma_2(t, y), \dots, \gamma_\mu(t, y) \) of the homological Milnor fibration of \( f_y \) for \( y \in I_\eta \), given by eigenvectors of the monodromy endomorphism, and they can be extended by parallel transport, to a basis of the homological bundle associated with the locally trivial fibration \( \Phi : \mathcal{X}' \longrightarrow T'_\delta \times I_\eta \).

\jump

After the above discussion, we are in the same situation as in \cite[Cor. 21]{varchenko80}, replacing the distinct eigenvalues of the monodromy by the distinct generalized eigenvectors of the monodromy. Therefore, assuming the existence of dual locally constant geometric sections, there is a single \( b \)-exponent \( \alpha \in L(\lambda) \) associated with each generalized eigenvector of the monodromy. The following proposition then follows from the same argument as in \cite[Cor. 21]{varchenko80}.

\begin{proposition} \label{prop:semicontinuity-bexponent}
Let \( \gamma_{\lambda}(t, y) \) be a generalized eigenvector of the monodromy of eigenvalue \( \lambda \in \Lambda \). Assume that there exist dual locally constant geometric sections to \( \gamma_{\lambda}(t, y) \) for all fibers of the deformation, that is, \( \dim_{\mathbb{C}} H^n_{\gamma_\lambda}(y) = 1 \) for all \( y \in I_\eta \). Then, the \( b \)-exponent associated with \( \gamma_{\lambda}(t, y) \) depends upper-semicontinuously on the parameter \( y \) of the \( \mu \)-constant deformation.
\end{proposition}
%\begin{proof}
%For the semicontinuity of the dimension \( d_{\gamma_\lambda, \alpha}(y) \), one argues as in the proof of \cref{prop:semicontinuity-dimension} given in \cite{varchenko80} since the bundle \( H^n_{\gamma_{\lambda, \alpha}} \) is a subbundle of \( H^n_{\lambda, \alpha} \). Then, after \cref{thm:varchenko} and since \( \dim_{\mathbb{C}} H^n_{\gamma_\lambda}(y) = 1 \) for all \( y \in I_\eta \), one has that \( \alpha \in L(\lambda) \) is a \( b \)-exponent of \( f_y \), if and only if, \( d_{\gamma_\lambda, \alpha}(y) - d_{\gamma_\lambda, \alpha - 1}(y) = 1 \).
%\end{proof}

\subsection{Resolution of singularities} \label{sec:resolution-sing}

In this section, we show how to use a resolution of singularities of \( f \) to study the integrals of relative differential forms along vanishing cycles following the ideas of Varchenko in \cite[\S 4]{varchenko82}.

\jump

Let \( \pi : \overbar{X} \longrightarrow X \) be a resolution of singularities of \( f \). This means that \( \overbar{X} \) is a smooth complex analytic manifold and \( \pi \) is a proper birational map that is an isomorphism outside the singular locus of \( f \). The total transform divisor and the relative canonical divisor are simple normal crossing divisors with the following expressions
\begin{equation} \label{eq:eq123}
F_\pi := \textnormal{Div}(\pi^* f) = \sum_{i=1}^r N_i E_i + \sum_{i=r+1}^s S_i, \qquad K_\pi := \textnormal{Div}(\det\textnormal{Jac}(\pi)) = \sum_{i=1}^r k_i E_i,
\end{equation}
where \( E_i \) and \( S_i \) are, respectively, the irreducible components of the exceptional divisor \( E \) and of the strict transform of \( f \). For reasons that will become clear later, it is convenient to pass from the resolution manifold \( \overbar{X} \) to another space where the normal crossings of the exceptional divisor \( E \) are preserved but \( F_\pi \) is \emph{reduced}. This is indeed possible if we relax the smoothness conditions of \( \overbar{X} \). This process is called \emph{semi-stable reduction} and the reader is referred to \cite[\S 2]{steenbrink77} for the details.

\jump

Let \( e \) be a positive integer such that the \( e \)--th power of the monodromy is unipotent. By A'Campo's description of the monodromy in terms of a resolution of \( f \) given in \cite{acampo75}, we can take \( e = \lcm(N_1, N_2, \dots, N_r) \). Define
\begin{equation*}
\widetilde{T} := \left\{\, \tilde{t} \in \mathbb{C} \ |\ |\tilde{t}| < \delta^{1/e}\, \right\},
\end{equation*}
and let \( \sigma : \widetilde{T} \longrightarrow T \) be given by \( \sigma(\tilde{t}) = \tilde{t}^e  \). Denote by \( \widetilde{X} \) the normalization of the fiber product \( \overbar{X} \times_T \widetilde{T} \) and by \( \nu : \widetilde{X} \longrightarrow \overbar{X} \times_T \widetilde{T} \) the normalization morphism. Let \( \rho : \widetilde{X} \longrightarrow \overbar{X} \) and \( \tilde{f} : \widetilde{X} \longrightarrow \widetilde{T} \) be the natural maps. Finally, denote \( \widetilde{X}_{\tilde{t}} := \tilde{f}^{-1}(\tilde{t}) \) and \( \widetilde{D} := \rho^*(D) \) for any divisor \( D \) on \( \overbar{X} \). We have the following commutative diagram
\begin{equation*}
  \begin{tikzcd}
    \widetilde{X} \arrow{d}{\tilde{f}} \arrow{r}{\rho} & \overbar{X} \arrow{d}{\pi^* f} \arrow{r}{\pi} & X \arrow{d}{f} \\
    \widetilde{T} \arrow{r}{\sigma} & T \arrow[r, equal] & T.
  \end{tikzcd}
\end{equation*}

\jump

An orbifold of dimension \( n + 1 \) is a complex analytic space which admits an open covering \( \{U_i\}  \) such that each \( U_i \) is analytically isomorphic to \( Z_i/G_i \) where \( Z_i \subset \mathbb{C}^{n+1}  \) is an open ball and \( G_i \) is finite subgroup of \( GL(n+1, \mathbb{C}) \). Similarly, a divisor \( D \) on an orbifold \( \widetilde{X} \) is an orbifold normal crossing divisor if locally \( ( \widetilde{X}, \widetilde{D}) = (Z, F)/G \) with \( Z \subset \mathbb{C}^{n+1}\) an open domain, \( G \subset GL(n+1, \mathbb{C}) \) a finite subgroup acting on \( Z \) and \( F \subset Z \) a \( G \)--invariant divisor with normal crossings. The singularities of an orbifold are concentrated in codimension at least two.

\begin{lemma}[{\cite[Lemma 2.2]{steenbrink77}}]
\( \widetilde{X} \) is an orbifold and the divisor of \( (\pi \rho)^*f \) is a reduced divisor with orbifold normal crossings.
\end{lemma}

\jump

In terms of the local coordinates of \( \overbar{X} \), the orbifold \( \widetilde{X} \) is presented as follows. Fix \( U \) an affine coordinate chart of \( \overbar{X} \) with coordinates \( x_0, \dots, x_n \), for which there are non-negative integers \( k \) and \( N_{i_0}, \dots, N_{i_k} \) such that \( (\pi^*{f})(x_0, \dots, x_n) = x_0^{N_{i_0}} \cdots x_{k}^{N_{i_k}} \). Then, on the open neighborhood \( U \times_{T} \widetilde{T} \) in \( \overbar{X} \times_{T} \widetilde{T} \), we have \( x_0^{N_{i_0}} \cdots x_k^{N_{i_k}} = \tilde{t}^e \). Set \( d := \gcd(N_{i_0}, \dots, N_{i_k}) \), the preimage of \( U \times_{T} \widetilde{T} \) in \( \widetilde{X} \) consists of \( d \) disjoint open sets which we denote by \( U_1, U_2, \dots, U_d \). On one of these subsets \( U_j \), there are coordinates \( y_0, \dots, y_k, \tau \) related by \( \tau = y_0 \cdots y_k \). The map \( \restr{\rho}{U_j} : U_j \longrightarrow U \times_{T} \widetilde{T} \) is given by \( \tilde{t} = \tau \exp(2 \pi \imath j/d) \) and \( x_i = y_i^{e/N_i} \) if \( 0 \leq i \leq k \) and \( x_i = y_i \) if \( i > k \).

\jump

Let \( G = \mathbb{Z}/(e/N_{i_0}) \times \cdots \times \mathbb{Z}/(e/N_{i_k}) \) be the group that acts on \( \mathbb{C}\{y_0, \dots, y_k\} \) according to the following rules
\begin{equation*}
(a_0, \dots, a_k) \cdot y_i = \begin{cases}
         \exp(2 \pi \imath a_j N_j / e) \cdot y_j  & \textnormal{if} \ 0 \leq j \leq k,\\
          y_j                              & \textnormal{if} \ j > k.
\end{cases}
\end{equation*}
Let \( G' := \{g \in G \ |\ g \tau = \tau\} \subset GL(n+1, \mathbb{C}) \). Then, the holomorphic functions and differential forms on \( U_j \) are the usual ones in terms of the coordinates \( y_0, \dots, y_n \) subject to the condition that they must be invariant under \( G' \), i.e. \( g \cdot (y_0 \cdots y_k) = y_0 \cdots y_k \). In this context, differential calculus is completely analogous to the usual one on manifolds.

\jump

Assume that \( \omega \in \Gamma(X, \Omega^{n+1}_X) \) is a top holomorphic form. Let \( v_i(\omega) \) be the order of vanishing of \( \pi^*\omega \) along the exceptional component \( E_i \), then the order of vanishing \( \tilde{v}_i(\omega) \) of \((\pi \rho)^* \omega \) along \( \widetilde{E}_i \) is \( e(v_i(\omega) + 1)/N_i - 1 \), see \cite[Lemma 4.4]{varchenko82}. Clearly after \cref{eq:eq123}, \( v_i(\omega) = k_i \) if \( \omega = \dd x_0 \wedge \cdots \wedge \dd x_n \), and thus \( \tilde{v}_i(\omega) = e(k_i + 1)/N_i - 1 \). Now, take \( \widetilde{\omega} \) a section of \( \Omega^{n+1}_{\widetilde{X}} \), since locally \( \tilde{f}(y_0, \dots, y_n) = y_0 \cdots y_k \), the relative form \( \widetilde{\omega} / \dd \tilde{f} \) is well-defined on
\begin{equation*}
\widetilde{F}_\pi^\circ := \bigcup_{i = 1}^r \widetilde{D}_i^\circ \quad \textnormal{where} \quad \widetilde{D}_i^\circ := \widetilde{D}_i \setminus \bigcup_{j \neq i} (\widetilde{D}_i \cap \widetilde{D}_j) \quad \textnormal{with} \quad D_i, D_j \in \textnormal{Supp}(F_\pi).
\end{equation*}

The following lemma is easy to establish.

\begin{lemma}[{\cite[Lemma 4.3]{varchenko82}}] \label{lemma:varchenko}
If \( \pi \rho \) also denotes the restriction to \( \widetilde{X}_{\tilde{t}}, \tilde{t} \in \widetilde{T}' \), of the map \( \pi \rho : \widetilde{X} \longrightarrow X \), then for all \(\tilde{t} \in \widetilde{T}' := \sigma^{-1}(T')\),
\begin{equation*}
\restr{\frac{(\pi \rho)^*(\omega)}{\dd \tilde{f}}}{\widetilde{X}_{\tilde{t}}} = e \tilde{f}^{e-1} (\pi \rho)^*\bigg( \restr{\frac{\omega}{\dd f}}{X_{\tilde{t}^{e}}} \bigg).
\end{equation*}
\end{lemma}

\subsection{The asymptotic expansion of the periods} \label{sec:filtrations}

In this section, we will construct the asymptotic expansion of the integrals in \cref{eq:integral-expansion-total} from a resolution of singularities of \( f \). We will use the same notations from the previous section. In the sequel, we will always fix an exceptional component \( E_i \) of a resolution of the germ \( f : (\mathbb{C}^{n+1}, \boldsymbol{0}) \longrightarrow (\mathbb{C}, 0) \).

\jump

Let \( \omega \in \Gamma(X, \Omega^{n+1}_X)  \) be a holomorphic differential form of maximal degree on \( X \).
Let \( \overbar{X}_i \subset \overbar{X} \) be an open neighborhood of the divisor \( E_i \) and define \( \overbar{X}_i^\circ := \overbar{X}_i \cap (\overbar{X} \setminus F_\pi) \). If locally on \( \overbar{X}_i \) the exceptional divisor \( E_i \) has local equation \( x_0 \), then one can locally decompose \( \overline{\omega} := \pi^*\omega \) in the following way
\[ \overline{\omega} = \overline{\omega}_0 + \overline{\omega}_1 + \cdots + \overline{\omega}_\nu + \cdots, \quad \nu \in \mathbb{Z}_{+}, \]
where \( \overline{\omega}_\nu \) is a local section of \( \Omega^{n+1}_{\overbar{X}}(-\nu E_i) \) and \( x_0^{-v_i(\overline{\omega}_\nu)} \overline{\omega}_\nu \) extends holomorphically to \( \overbar{X}_i \), by expanding the local defining equation of \( \overline{\omega} \) as a series in \( x_0 \). Let us see that these local sections patch up to a section over \( \overbar{X}_i \).
\begin{lemma}
For \( \nu \in \mathbb{Z}_{+} \), the local forms \( \overline{\omega}_{\nu} \) define an element of \( \Gamma\big(\overbar{X}_i, \Omega_{\overbar{X}}^{n+1}(-\nu E_i)\big) \).
\end{lemma}
\begin{proof}
Let \( \overline{\omega}_{\nu, \alpha} \) and \( \overline{\omega}_{\nu, \beta} \) be two holomorphic forms constructed as above and defined on the open sets \( U_{\alpha}, U_{\beta} \subset \overbar{X}_i \). We must check that these local sections agree on the intersection \( U_\alpha \cap U_\beta \).

\jump

Let \( x_0 \) and \( y_0 \) be local equations of \( E_i \) on \( U_\alpha \) and \( U_\beta \), respectively. Since \( E_i \) is a Cartier divisor, on \( U_\alpha \cap U_\beta \) it must hold that \( x_0 = u_{\alpha, \beta} y_0 \) for some \( u_{\alpha, \beta} \in \Gamma(U_\alpha \cap U_\beta, \mathcal{O}^*_{\overbar{X}}) \). Consider \( x_1, \dots, x_n \) and \( y_1, \dots, y_n \) local coordinates on \( U_\alpha \cap E_i \) and \( U_\beta \cap E_i \), respectively. Finally, let \( \varphi_{\alpha, \beta} \) and \( \varphi_{\beta, \alpha} \) the corresponding coordinate transformations on \( U_\alpha \cap U_\beta \cap E_i \).

\jump

The first observation is that, as a function of \( y_0, \dots, y_n \), \( u_{\alpha, \beta} \) does not depend on the variable \( y_0 \). Indeed, it holds that \( y_0 = u_{\beta, \alpha} x_0 \) with \( u_{\beta, \alpha} \in \Gamma(U_\alpha \cap U_\beta, \mathcal{O}^*_{\overbar{X}}) \) and \( u_{\alpha, \beta} u_{\beta, \alpha} = 1 \). Then from \( x_0 = u_{\alpha, \beta} y_0 \), the only way that the functional inverse of \( y_0 \) is equal to \( u_{\alpha, \beta}^{-1} x_0 \) is that \( u_{\alpha, \beta} \) does not depend on \( y_0 \).

\jump

Assume that in these local coordinates,
\[ \restr{\overline{\omega}}{U_\alpha} = x_0^{v_i(\omega)} h_{\alpha}(x_0, \dots, x_n) \dd \underline{x} \quad \textnormal{and} \quad \restr{\overline{\omega}}{U_\beta} = y_0^{v_i(\omega)} h_{\beta}(y_0, \dots, y_n) \dd \underline{y}. \]
By construction, we will then have that
\[ \overline{\omega}_{\nu, \alpha} = x_0^{v_i(\omega) + \nu} h_{\alpha,\nu}(x_1, \dots, x_n) \dd \underline{x} \quad \textnormal{and} \quad \overline{\omega}_{\nu, \beta} = y_0^{v_i(\omega)+\nu} h_{\beta, \nu}(y_1, \dots, y_n) \dd \underline{y}, \]
with
\begin{equation} \label{eq:lemma-section-eq3}
  h_{\beta, \nu}(y_1, \dots, y_n) = \frac{1}{\nu!} \left(\frac{\partial^\nu}{\partial y_0^\nu}\cdot h_\beta(y_0, \dots, y_n)\right)\bigg|_{y_0=0} \quad \textnormal{on} \quad U_\beta,
\end{equation}
and a similar expression holds for \( h_{\alpha, \nu} \) on \( U_{\alpha} \). To see that \( \restr{\overline{\omega}_{\nu, \alpha}}{U_\alpha \cap U_\beta} = \restr{\overline{\omega}_{\nu, \beta}}{U_\alpha \cap U_\beta} \), it is enough to show that
\begin{equation} \label{eq:lemma-section-eq1}
  h_{\beta, \nu}(y_1, \dots, y_n) = u_{\alpha, \beta}^{v_i(\omega) + \nu} h_{\alpha, \nu}(\varphi_{\alpha, \beta})\, \Phi_{\alpha, \beta} \quad \textnormal{on} \quad U_\alpha \cap U_\beta,
\end{equation}
where \( \dd \underline{x} = \Phi_{\alpha, \beta} \dd \underline{y} \) and we have that \( \Phi_{\alpha, \beta} = u_{\alpha, \beta}\, \textnormal{det}\, \textnormal{Jac}(\varphi_{\alpha, \beta}) \) does not depend on \( y_0 \). Since \( \overline{\omega} \) is a global section on \( \overbar{X} \), one has that \( \restr{\overline{\omega}}{U_\alpha} = \restr{\overline{\omega}}{U_\beta} \) on \( U_\alpha \cap U_\beta\). Hence,
\begin{equation} \label{eq:lemma-section-eq2}
h_{\beta}(y_0, \dots, y_n) = u_{\alpha, \beta}^{v_i(\omega)} h_{\alpha}(u_{\alpha, \beta} y_0, \varphi_{\alpha, \beta})\, \Phi_{\alpha, \beta} \quad \textnormal{on} \quad U_\alpha \cap U_\beta.
\end{equation}
Finally, to show that \eqref{eq:lemma-section-eq1} holds, it is enough to substitute \cref{eq:lemma-section-eq2} in \eqref{eq:lemma-section-eq3} and apply the chain rule together with the fact that \( u_{\alpha, \beta} \) does not depend on \( y_0 \).
\end{proof}

We will call \( \overline{\omega}_\nu \) the \( \nu \)--th piece of \( \omega \) associated with the divisor \( E_i \). For the ease of notation, we will omit the dependence of the pieces \( \overline{\omega}_{\nu} \) on the index \( i \) as we will always work with a fixed divisor \( E_i \).

\jump

Set \( \widetilde{X}^\circ_i := \nu^{-1}\big(\overbar{X}_i^\circ \times_{T} \widetilde{T}\big) \subset \widetilde{X} \), where \( \nu \) is the normalization morphism from the previous section. Recall that on \( \widetilde{X}^\circ_i \) and \( \tilde{T} \) there are local coordinates such that \( \tilde{f}(y_0, \dots, y_n) = y_0 = \tilde{t} \) with \( \widetilde{E}^\circ_i : y_0 = 0 \). Notice also that if \( \widetilde{\omega}_\nu := \rho^* \overline{\omega}_\nu \), then the orders of vanishing have the following relation: since \( v_i(\overline{\omega}_\nu) = v_i(\omega) + \nu \), then \( \tilde{v}_i(\widetilde{\omega}_\nu) = e(v_i(\omega) + 1 + \nu)/N_i - 1 \). With all these considerations, the form \(\tilde{f}^{-\tilde{v}_i(\widetilde{\omega}_\nu)} \widetilde{\omega}_\nu\), holomorphic on \( \widetilde{X}_i^\circ \), extends holomorphically over \( \widetilde{E}_i^\circ \), and the next lemma follows.

\begin{lemma} \label{lemma:integral-valuations}
Let \( \tilde{\gamma}(\tilde{t}) \) be any \( n \)-cycle on \( \widetilde{X}^\circ_i \cap \tilde{f}^{-1}(\tilde{t}), \tilde{t} \in \tilde{T}' \). If \( \widetilde{\omega} := \rho^* \overline{\omega} \), then
\begin{equation} \label{eq:integral-valuations}
\int_{\tilde{\gamma}(\tilde{t})} \frac{\widetilde{\omega}}{\dd \tilde{f}} = \sum_{\nu \geq 0} \tilde{t}^{\, \tilde{v}_i(\widetilde{\omega}_\nu)} \int_{\tilde{\gamma}(\tilde{t})} R_i(\widetilde{\omega}_\nu),
\end{equation}
where \( R_i(\widetilde{\omega}_\nu) := \tilde{f}^{-\tilde{v}_i(\widetilde{\omega}_\nu)} \widetilde{\omega}_\nu / \dd \tilde{f} \) is defined as in \eqref{eq:divide_df}. The \( n \)-form \( R_i(\widetilde{\omega}_\nu) \) is a well-defined holomorphic form on \(\widetilde{X}_i^\circ\) that extends holomorphically over \( \widetilde{E}_i^\circ\).
\end{lemma}

Notice that since \( \overline{\omega}_0 \) is always different from zero, \( R_i(\widetilde{\omega}_0) \) is always a non-zero \( n \)-form. However, this may not be the case for the other terms \( R_i(\widetilde{\omega}_\nu), \nu > 0 \).

\jump

Now we can obtain the expression of \cref{eq:integral-expansion-total} in terms of the resolution data by pushing down to \( X \) the expressions from \cref{eq:integral-valuations}. Namely, after \cref{lemma:varchenko}, the left-hand side of \cref{eq:integral-valuations} reads as
\[ \int_{\tilde{\gamma}(\tilde{t})} \frac{(\pi \rho)^* \omega}{\dd \tilde{f}} = e \tilde{t}^{e - 1}\int_{\tilde{\gamma}(\tilde{t})}(\pi \rho)^* \Big( \frac{\omega}{\dd f} \Big). \]
Define the numbers
\begin{equation} \label{eq:def-sigma}
  \sigma_{i, \nu}(\omega) := \frac{v_i(\omega) + 1 + \nu}{N_i}, \qquad \nu \in \mathbb{Z}_{+}.
\end{equation}
In particular, \( \sigma_{i, \nu}(\dd x_0 \wedge \dots \wedge \dd x_n) = (k_i + 1 + \nu)/N_i \). Then
\[ \int_{\tilde{\gamma}(\tilde{t})} (\pi \rho)^* \Big( \frac{\omega}{\dd f} \Big) = e^{-1} \sum_{\nu \geq 0} \tilde{t}^{e(\sigma_{i, \nu}(\omega) - 1)} \int_{\tilde{\gamma}(\tilde{t})} R_{i}(\widetilde{\omega}_\nu ), \]
since \( \tilde{v}_i(\widetilde{\omega}_{\nu}) - e + 1 = e(\sigma_{i, \nu}(\omega) - 1) \). Finally, since \( \tilde{t}^e = t \), the following lemma follows.

\begin{lemma} \label{lemma:integral-valuations2}
For any locally constant \( n \)-cycle \(\tilde{\gamma}(\tilde{t}) \) in \( H_n\big(\widetilde{X}^\circ_i \cap \tilde{f}^{-1}(\tilde{t}), \mathbb{C}\big) \), consider \( \gamma(\tilde{t}^e) := \rho_*\tilde{\gamma}(\tilde{t}) \), then
\begin{equation} \label{eq:integral-valuations2}
  \int_{\pi_*\gamma({t})} \frac{\omega}{\dd f} = \sum_{\nu \geq 0} t^{\sigma_{i, \nu}(\omega) - 1} \int_{\gamma(t)} R_{i, \nu}(\omega),
\end{equation}
where \( R_{i, \nu}(\omega) := e^{-1} (\rho^{-1})^* R_i(\widetilde{\omega}_\nu) \) is an \( n \)-form on \( \overbar{X}_i^\circ \) that extends over \( E_i^\circ \) and that does not depend on the integer \( e \).
\end{lemma}

The missing logarithmic terms in \cref{eq:integral-valuations2}, when compared to \cref{eq:integral-expansion-total}, are in the integrals on the right-hand side of \cref{eq:integral-valuations2}. Indeed, some of these integrals may blow-up to infinity as \( t \) tends to zero. This would mean that, after \cref{eq:integral-expansion-total}, there is a logarithmic term associated with the exponent \( \sigma_{i, \nu}(\omega) - 1 \).

\jump

For the particular case where \( \omega = \dd x_0 \wedge \dots \wedge \dd x_n \), by comparing \cref{lemma:integral-valuations2} with \cref{eq:integral-expansion-total} and using \cref{thm:varchenko} one obtains the following classical result of Kashiwara \cite{kashiwara76} and Lichtin \cite{lichtin86} for the case of isolated singularities.
\begin{theorem}
Let \( f:(\mathbb{C}^{n+1}, \boldsymbol{0}) \longrightarrow (\mathbb{C}, 0) \) be a germ of a holomorphic function defining an isolated singularity. Then, the \( b \)-exponents of \( f \) have the form
\[ \frac{k_i + 1 + \nu}{N_i}, \]
for some \( \nu \in \mathbb{Z}_{+}\) and some \( i = 1, \dots, r.\)
\end{theorem}

\jump

Let \( D_{i, j} \) be the set-theoretic intersection of \( E_i \) with another irreducible component \( D_j \in \textrm{Supp}(F_\pi) \). Then, \( D_{i, j} \) are divisors on \( E_i \) since \( F_\pi \) is a simple normal crossing divisor. By definition \( E_i^\circ := E_i \setminus \cup_j D_{i, j} \).

\begin{lemma}
The restriction of the \( n \)-form \( R_{i, \nu}(\omega) \) to \( E_i^\circ \) is multivalued with order of vanishing along the divisor \( D_{i, j} \) equal to
\begin{equation} \label{eq:definition-epsilons}
  \varepsilon_{j, \nu}(\omega) := -N_j \sigma_{i, \nu}(\omega) + v_j(\overline{\omega}_\nu) = -N_j \frac{v_i(\omega) + 1 + \nu}{N_i} + v_j(\overline{\omega}_\nu).
\end{equation}
That is, \(R_{i, \nu}(\omega) \) defines an element of \( \Gamma(E_i^\circ, \Omega^n_{E_i^\circ} \otimes L) \) where \( L \) is the local system on \( E_i^\circ \) having monodromy \( e^{-2 \pi \imath \varepsilon_{j, \nu}(\omega)} \) around the divisor \( D_{i, j} \).
\end{lemma}
\begin{proof}
The lemma follows from a local computation. Let \( x_0, x_1 \) be local coordinates around a general point of \( D_{i, j} \), such that \( E_i : x_0 = 0 \), \(D_{j} : x_1 = 0 \) and \( f = x_0^{N_i} x_1^{N_j} \). Then, for any \( \omega \in \Gamma(X, \Omega^{n+1}_X) \) and \( \nu \in \mathbb{Z}_{+} \),
\[ \overline{\omega}_\nu = x_0^{v_i(\omega) + \nu} x_1^{v_j(\overline{\omega}_\nu)} v(x_1, \dots, x_n) \dd x_0 \wedge \dots \wedge \dd x_n. \]
Following \cref{sec:resolution-sing}, we set \( x_0 = y_0^{e/N_i}, x_1 = y_1^{e/N_j} \) and \( x_k = y_k \) otherwise. Hence,
\[ \widetilde{\omega}_\nu = \frac{e}{N_i} \frac{e}{N_j} y_0^{e(v_i(\omega) + 1 + \nu)/N_i - 1} y_1^{e(v_j(\overline{\omega}_\nu) + 1)/N_j - 1} v \dd y_0 \wedge \dots \wedge \dd y_n, \]
and \( f = y_0 y_1 \). Now, on \( \widetilde{E}_i^{\circ} \), \( f = \bar{y}_0 \) with \( y_0 = \bar{y}_0/y_1 \). Making the substitution on \( \widetilde{\omega}_\nu \),
\[ \widetilde{\omega}_\nu = \frac{e}{N_i} \frac{e}{N_j} \bar{y}_0^{e \sigma_{i, \nu}(\omega) - 1} y_1^{e(v_j(\overline{\omega}_\nu) + 1)/N_j - e \sigma_{i, \nu}(\omega) - 1} v \dd \bar{y}_0 \wedge \dots \wedge \dd y_n. \]
Now, with the notations from \cref{lemma:integral-valuations}, \( R_i(\widetilde{\omega}_\nu) \) is given locally around \( y_1 = 0 \) on \( \widetilde{E}_i^{\circ} \) by the expression
\[ \frac{e}{N_j} y_1^{e(v_j(\overline{\omega}_\nu) + 1)/N_j - e\sigma_{i, \nu}(\omega) - 1} v \dd y_1 \wedge \dots \wedge \dd y_n. \]
Finally, to get the local expression for the \( R_{i, \nu}(\omega)  \) from \cref{lemma:integral-valuations2}, one simply undoes the first change of variables, i.e. \( y_1 = x_1^{N_j/e} \) and \( y_k = x_k \) otherwise, obtaining
\[ x_1^{v_j(\overline{\omega}_\nu) - N_j \sigma_{i, \nu}(\omega)} v \dd x_1 \wedge \dots \wedge \dd x_n, \]
as we wanted to show.
\end{proof}

Since no confusion may arise we drop the dependency on the index \( i \) of the divisor \( E_i \) when referring to the numbers \( \varepsilon_{j, \nu}(\omega) \).

\jump

In the next sections we will give conditions in the case \( n = 1 \) for \( R_{i, \nu}(\omega) \) to define a (non-zero) locally constant cohomology class \( A^\omega_{\sigma_{i, \nu}-1, 0}(t) \). Precisely, this means that the pairing defined by
\begin{equation} \label{eq:pairing}
\langle A_{\sigma_{i, \nu} - 1, 0}^\omega(t), \pi_* \gamma(t) \rangle := \lim_{t \rightarrow 0} \int_{\gamma(t)} R_{i, \nu}(\omega)
\end{equation}
is well-defined for any locally constant cycle \( \gamma(t) \) on \( H_n(\overbar{X}_t, \mathbb{C}) \). Here, and in the sequel, we denote \( A^\omega_{\sigma_{i, \nu} -1, 0} \) instead of \( A^\omega_{\sigma_{i, \nu}(\omega) -1, 0} \) for the ease of notation. Notice also that, after \cref{lemma:integral-valuations2}, the pairing \eqref{eq:pairing} coincides with the pairing from \cref{sec:geometric-sections} with \( \alpha = \sigma_{i, \nu} \) and that the locally constant sections defined from the forms \( R_{i, \nu}(\omega) \) are such that \( s_{\alpha-1}[A^\omega_{\alpha-1, 0}] = t^{\alpha-1} A^\omega_{\alpha-1, 0}(t). \)

\jump

We will end this section with some sufficient conditions for a certain \( \sigma_{i, \nu}(\omega) \) associated to a differential form \( \omega \in \Gamma(X, \Omega^{n+1}_X) \) and an exceptional divisor \( E_i \) to be a \( b \)-exponents.
\begin{lemma} \label{lemma:b-exponents}
Let \( \omega \in \Gamma(X, \Omega^{n+1}_X) \) and \( \nu \in \mathbb{Z}_{+} \) such that \( {\overline{\omega}_{\nu}} \) is non-zero. Assume that \( R_{i, \nu}(\omega) \) gives a non-zero locally constant cohomology class \( A^\omega_{\sigma_{i, \nu}-1, 0}(t) \) via the pairing in \eqref{eq:pairing} and that \( A_{\sigma_{i, \nu}-1, 0}^{\omega} \) is not a section of \( S_{\sigma_{i, \nu}(\omega) - 2} \). Then, \( \sigma_{i, \nu}(\omega) \) is a \( b \)-exponent of \( f \).
\end{lemma}
\begin{proof}
Set \( \alpha := \sigma_{i, \nu}(\omega) \). To deduce from \cref{thm:varchenko} that \( \alpha \) is a \( b \)-exponent, one should check that the class of \( s_{\alpha-1}[A^\omega_{\alpha-1, 0}] \) is non-zero in \( \mathcal{G}_{\alpha-1} \). Assume that it is zero, that is \( s_{\alpha-1}[A^\omega_{\alpha-1, 0}] = s_{\alpha -2}[A^\eta_{\alpha-2, 0}] \neq 0, \) for some \(\eta \in \Gamma(X, \Omega_{X}^{n+1}) \). In this case this means \( t^{\alpha-1} A_{\alpha-1, 0}^\omega(t) = t^{\alpha-2} A_{\alpha-2, 0}^\eta(t) \). Hence, \( A_{\alpha -1, 0}^\omega(t) \) is a section of \( S_{\alpha - 2} \).
\end{proof}

\jump

We will devote the rest of this work to check these conditions, under some hypothesis of genericity, for the candidate \( b \)-exponents of irreducible plane curve singularities.

\section{The case of plane curve singularities}

\subsection{Yano's conjecture} \label{sec:yanos-conjecture}

Next, we will present the conjecture posed in 1982 by Yano \cite{yano82} about the generic \( b \)-exponents of irreducible germs of plane curve singularities. The conjecture predicts that for generic curves in some \( \mu \)-constant deformation of \( f \), the whole set of \(\mu\) \( b \)-exponents can be completely determined from the characteristic sequence of the topological class.

\jump

Accordingly, let \( f : (\mathbb{C}^2, \boldsymbol{0}) \longrightarrow (\mathbb{C}, 0) \) be a germ of a holomorphic function defining an irreducible plane curve singularity with characteristic sequence \( (n, \beta_1, \dots, \beta_g) \), where \( n \) is the multiplicity of \( f \) at the origin and \( g \geq 1 \) is the number of characteristic pairs. With the same notations as \cite[\S 2]{yano82}, define
\begin{equation} \label{eq:yano-definitions}
\begin{split}
& e_{0} := n, \quad e_{i} := \gcd(n, \beta_1, \dots, \beta_i), \qquad i = 1, \dots, g, \\
& r_i := \frac{\beta_i + n}{e_{i}}, \quad R_i := \frac{\beta_i e_{i-1} + \beta_{i-1}(e_{i-2} - e_{i-1}) + \cdots + \beta_1(e_{0} - e_{1})}{e_{i}}, \\
& r'_0 := 2, \quad r'_i := r_{i-1} + \left\lfloor\frac{\beta_i - \beta_{i-1}}{e_{i-1}}\right\rfloor + 1 = \left\lfloor\frac{r_i e_{i}}{e_{i-1}}\right\rfloor + 1, \\
& R'_0 := n, \quad R'_i := R_{i-1} + \beta_i - \beta_{i-1} = \frac{R_i e_{i}}{e_{i-1}}.
\end{split}
\end{equation}
Inspired by A'Campo formula \cite{acampo75} for the eigenvalues of the monodromy of an isolated singularity, Yano defines the following polynomial with fractional powers in \( t \)
\begin{equation} \label{eq:yano-generating-series}
R\big((n, \beta_1, \dots, \beta_g), t\big) := \sum_{i = 1}^g t^{\frac{r_i}{R_i}} \frac{1 - t}{1 - t^{\frac{1}{R_i}}} - \sum_{i = 0}^g t^{\frac{r'_i}{R'_i}} \frac{1 - t}{1 - t^{\frac{1}{R'_i}}} + t,
\end{equation}
and proves that \( R\big((n, \beta_1, \dots, \beta_n), t\big) \) has non-negative coefficients. Finally, the conjecture reads as follows.
\begin{conjecture}[Yano \cite{yano82}]
Let \( f : (\mathbb{C}^2, \boldsymbol{0}) \longrightarrow (\mathbb{C}, 0) \) be a germ of a holomorphic function defining an irreducible plane curve. Then, for generic curves in some \( \mu \)-constant deformation of \( f \), the \( b \)-exponents \( \alpha_1, \dots, \alpha_\mu \) are given by the generating function \( R \). That is,
\begin{equation} \label{eq:yano-conj}
  \sum_{i = 1}^\mu t^{\alpha_i} = R\big((n, \beta_1, \dots, \beta_g), t\big).
\end{equation}
\end{conjecture}
The remaining part of this work will be devoted to prove Yano's conjecture.

\begin{remark}
The original formulation of the conjecture in \cite{yano82} expects a modality-dimensional \( \mu \)-constant deformation. The modality part is not relevant and has not been considered in the previous works on Yano's conjecture, see \cite{cassou-nogues88} and \cite{ABCNLMH16}. Indeed, this can already be seen in the examples \( W_{1, 2n-3} \) considered by Yano in \cite[\S 1]{yano82}. These examples correspond to irreducible plane curves with characteristic sequences \( (4, 6, 2n - 3), n \geq 5 \) and modality \( 2 \). However, there is only one curve with each of these topological types up to analytic isomorphism, see \cite[Chap. II, \S1]{teissier-appendix}. Hence, the \( b \)-exponents are constant along the fibers of any \( \mu \)-constant deformation of such plane branches.
\end{remark}

\subsection{Multivalued forms on the punctured projective line} \label{sec:multivalued}

In this section, we will review the basic facts from Section 2 of \cite{deligne-mostow86} about multivalued holomorphic forms on the punctured projective line defining cohomology classes in the cohomology groups with coefficients on a local system.

\jump

Let \( \mathbb{P} := \mathbb{P}^1_{\mathbb{C}} \) be the complex projective line, and \( S := \{s_1, s_2, \dots, s_r\}  \) be a set of \( r \geq 1\) distinct points on \( \mathbb{P} \) and \( (\alpha_s)_{s \in S} \) be a family of complex numbers satisfying \( \prod_{s \in S} \alpha_s = 1 \). With these data there is, up to non-unique isomorphism, a unique local system \( L \) of rank one in \( \mathbb{P} \setminus S \) such that the monodromy of \( L \) around each \( s \in S \) is the multiplication by \( \alpha_s \). We will denote by \( L^{\vee} \) the dual local system with monodromies \( \alpha_s^{-1}, s \in S \).

\jump

In order to work with the locally constant sections of \( L \), we fix complex numbers \( (\mu_s)_{s \in S} \) such that \( \alpha_s = e^{2 \pi \imath \mu_s}, s \in S \). Let \( z \) be a local coordinate near \( s \in S \). Any local section \( u \) of \( \mathcal{O}(L) := \mathcal{O}_{\mathbb{P} \setminus S} \otimes_{\underline{\mathbb{C}}} L\) (resp. \(\Omega^1(L) := \Omega^1_{\mathbb{P} \setminus S} \otimes_{\underline{\mathbb{C}}} L\)) in a neighborhood of \( s \) can be written as \( u = z^{-\mu_s} e f \) (resp. \( u = z^{-\mu_s} e f \dd z \)) with \( e \) a non-zero multivalued section of \( L \) and \( f \) holomorphic in a punctured neighborhood of \( s \). We define \( u \) to be meromorphic in \( s \in S \) if \( f \) is, and we define the valuation of \( u \) at \( s \) as \( v_s(u) := v_s(f) - \mu_s \).

\jump

The holomorphic \( L \)-valued de Rham complex \( \Omega^\bullet(L) \): \(\mathcal{O}(L) \longrightarrow \Omega^1(L) \) with the natural connecting morphism \( \dd(e f) = e \dd f \) is a resolution of \( L \) by coherent sheaves. Therefore, one can interpret \( H^*(\mathbb{P} \setminus S, L) \) as the hypercohomology on \( \mathbb{P} \setminus S \) of \(\Omega^\bullet(L)\). Since \( \mathbb{P} \setminus S \) is Stein, \( H^q(\mathbb{P} \setminus S, \Omega^p(L)) = 0 \) for \( q > 0 \) and the hypercohomology \( \mathbb{H}^*(\mathbb{P} \setminus S, \Omega^\bullet(L)) \) gives
\[ H^*(\mathbb{P} \setminus S, L) = H^* \Gamma(\mathbb{P} \setminus S, \Omega^\bullet(L)). \]
Let \( j : \mathbb{P} \setminus S \longrightarrow \mathbb{P} \) be the inclusion. Similarly, since \( j \) is a Stein morphism, the higher direct images \( {R}^q j_* \Omega^p(L) \) vanish for \( q > 0 \) and
\[ H^*(\mathbb{P} \setminus S, L) = H^* \Gamma(\mathbb{P} \setminus S, \Omega^\bullet(L)) = H^* \Gamma(\mathbb{P}, j_* \Omega^\bullet(L) ) = \mathbb{H}^*(\mathbb{P}, j_* \Omega^\bullet(L)). \]
It is convenient to replace the complex of sheaves \( j_* \Omega^\bullet(L) \) by the subcomplex \( j_*^m \Omega^p(L) \) of meromorphic forms. The analytic Atiyah-Hodge lemma implies that
\[ \mathbb{H}^*(\mathbb{P}, j_*^m \Omega^\bullet(L)) \cong \mathbb{H}^*(\mathbb{P}, j_* \Omega^\bullet(L)) = H^*(\mathbb{P} \setminus S, L). \]
Since \( j_*^m\Omega^p(L) \) is an inductive limit of line bundles with degrees tending to infinity, one can show that \( H^q(\mathbb{P}, j_*^m \Omega^p(L)) \) vanishes for \( q > 0 \). Indeed, if \( D = \sum_{s \in S} s \) is the divisor on \( \mathbb{P} \) associated with \( S \), then
\[
H^q(\mathbb{P}, j_*^m\Omega^q(L)) = H^q(\mathbb{P}, \lim_{\longrightarrow n} \Omega^p(L) \otimes \mathcal{O}(nD) ) = \lim_{\longrightarrow n} H^q(\mathbb{P}, \Omega^p(L) \otimes \mathcal{O}(nD)) = 0,
\]
since \( H^q(\mathbb{P}, \Omega^p(L) \otimes \mathcal{O}(nD)) \) vanishes for \( n \gg 0 \). Finally, this means that the cohomology groups \( H^q(\mathbb{P} \setminus S, L) \) can be computed as the cohomology of the complex of \( L \)--valued forms on \( \mathbb{P} \) meromorphic along \( D \),
\begin{equation} \label{eq:cohomology}
H^*(\mathbb{P} \setminus S, L) \cong H^* \Gamma(\mathbb{P}, j_*^m \Omega^\bullet(L) ).
\end{equation}

Define the line bundle \( \mathcal{O}(\sum \mu_s s)(L) \) as the subsheaf of \( j_*^m \mathcal{O}(L) \) whose local holomorphic sections are the local sections \( u \) of \( j_*^m \mathcal{O}(L) \) such that the integer \( v_s(u) + \mu_s \) is greater or equal than zero, i.e. \( v_s(u) \geq - \mu_s \). If \( u \) is a meromorphic section of \( \mathcal{O}(\sum_{s} \mu_s s)(L) \) and if \( s \in S \), one has that \( \deg_s(u) = v_s(u) + \mu_s \). The same holds for \( x \in \mathbb{P} \setminus S \) if one defines \( \mu_x = 0 \) for \( x \in \mathbb{P} \setminus S \). By \cite[Prop. 2.11.1]{deligne-mostow86}, the degree of the line bundle \( \mathcal{O}(\sum \mu_s s)(L) \) is equal to \( \sum_{s \in S} \mu_s \). Similarly, \( \deg \Omega^1(\sum \mu_s s)(L) = \sum_{s \in S} \mu_s - 2 \).

\jump

The following proposition is a slight generalization of \cite[Prop. 2.14]{deligne-mostow86}, where the differential form is assumed to be invertible in \( \mathbb{P} \setminus S \). We will allow \( \omega \) to have zeros in \( \mathbb{P} \setminus S \). Denote by \( \delta_x \in \mathbb{N} \) the order of vanishing of \( \omega \) in the points \( x \in \mathbb{P} \setminus S \).

\begin{proposition} \label{prop:deligne-mostow}
Let \( \omega \in \Gamma(\mathbb{P}, \Omega^1(\sum \mu_s s - \sum \delta_x x)(L)) \). Assume that \( \sum_{s \in S} \mu_s \leq r - 1 \) and that \( \alpha_s \neq 1 \) for all \( s \in S \). Then, \( \omega \) defines a non-zero cohomology class in \( H^1(\mathbb{P} \setminus S, L) \).
\end{proposition}
\begin{proof}
After \cref{eq:cohomology}, we want to show that equation \( \dd u = \omega \) is impossible for \( u \in \Gamma(\mathbb{P}, j_*^m \mathcal{O}(L)) \). For any section of \( j_*^m \mathcal{O}(L) \) verifying the relation, one has that \( v_x(\omega) \geq v_x(u) - 1 \), for \( x \in \mathbb{P} \). The equality may fail if \( v_x(u) = 0 \) or if \( v_x(\omega) \) is a non-negative integer. In any case, \( v_x(u) \) is always a non-negative integer if \( x \in \mathbb{P} \setminus S \). This implies that \( u \) must be a global section of the line bundle \( \mathcal{O}(\sum (\mu_s - 1) s - \sum \delta'_x x)(L)\) with \( \delta'_x > 0, x \in \mathbb{P} \setminus S \). But the degree of this line bundle is
\[ \sum_{s \in S}({\mu_s} - 1) - \sum_{x \in \mathbb{P} \setminus S} \delta'_x \leq \sum_{s \in S} \mu_s - r \leq -1, \]
which is impossible.
\end{proof}

\subsection{Periods of plane curve singularities} \label{sec:plane-curves}

In the sequel, let \( f :(\mathbb{C}^2, \boldsymbol{0}) \longrightarrow (\mathbb{C}, 0) \) be a germ of a holomorphic function defining a reduced plane curve singularity, not necessarily irreducible. Using the notations from \cref{sec:resolution-sing} we will fix any resolution \( \pi : \overbar{X} \longrightarrow X \) of \( f \). The exceptional divisor of any such resolution is composed exclusively of rational curves, i.e. \( E_i \cong \mathbb{P}^1_{\mathbb{C}} \). The irreducible components are usually classified in terms of the Euler characteristic \( \chi(E_i^\circ) \). Those exceptional components such that \( \chi(E_i^\circ) = 2 - r < 0 \) are usually called \emph{rupture} divisors and they will play a crucial role in the sequel.

\jump

Using the same notations from Sections \ref{sec:resolution-sing} and \ref{sec:filtrations}, let us fix an exceptional divisor \( E_i \) from a resolution of \( f \). Assume that one has a holomorphic 2-form \( \omega \in \Gamma(X, \Omega^{2}_{X}) \) such that the piece \( \overline{\omega}_\nu \) of degree \( \nu \) associated with the \( i \)--th rupture divisor \( E_i \) is non-zero. The following argument to find cycles \( C \) on rupture exceptional divisors such that for certain candidate exponents \( \sigma_{i, \nu} - 1 \) one has that
\begin{equation} \label{eq:construction-loeser}
  \lim_{t \rightarrow 0} \int_{\gamma(t)} R_{i, \nu}(\omega) = m \int_{C} R_{i, \nu}(\omega) \neq 0, \quad m \in \mathbb{Z},
\end{equation}
is basically due to Loeser \cite[\S III.3]{loeser88} using the results of Deligne and Mostow reviewed in \cref{sec:multivalued}. Thus, as noticed earlier in \cref{sec:filtrations}, this implies that the multivalued form \( R_{i, \nu}(\omega) \) on \( E_i^\circ \) defines a non-zero locally constant section \( A_{\sigma_{i, \nu} - 1, 0}^\omega \).

\jump

In the case of plane curves, the divisors \( D_{i, j} \) on \( E_i \) are just points which we will denote by \( p_j \), dropping its dependence on \( E_i \) since no confusion may arise. Let
\[ S_{i, \nu}(\omega) := \{ p_j \in E_i \ |\ p_j = E_i \cap D_j\ \textrm{with}\ D_j \in \textrm{Supp}(F_\pi)\ \textrm{and}\ \varepsilon_{j, \nu}(\omega) \neq 0 \}, \]
and let \( L \) be the local system on \( E_i \setminus S_{i, \nu}(\omega)\) with monodromies \( e^{-2 \pi \imath \varepsilon_{j, \nu}(\omega)} \) at the points \( p_j \in S_{i, \nu}(\omega) \). The forms \( R_{i, \nu}(\omega) \), for \( \nu > 0 \), might not be invertible in \( E_i^{\circ} \). Hence, denote by \( q_k \in E_i^\circ, k = 1, \dots, r \) the points where \( R_{i, \nu}(\omega) \) has zeros of order \( \delta_{k, \nu}(\omega) > 0\). Then, the multivalued form \( R_{i, \nu}(\omega) \) defines an element of \(\Gamma(E_i, \Omega^1_{E_i}(-\sum \varepsilon_{j, \nu}(\omega) p_j - \sum \delta_{k, \nu}(\omega) q_k)(L))\) in the sense of \cref{sec:multivalued}.

\jump

The following lemma is key to applying the results of Deligne and Mostow from \cref{sec:multivalued}. Other versions of this result in the case \( \nu = 0 \) can be found in the works of Lichtin \cite{lichtin85} and Loeser \cite{loeser88}.

\begin{proposition} \label{prop:epsilons}
For any holomorphic form \( \omega \in \Gamma(X, \Omega^2_X) \),
\[ \sum_{j = 1}^r \varepsilon_{j, \nu}(\omega) + \sum_{k = 1}^{s} \delta_{k, \nu}(\omega) = -2 - \nu E_i^2. \]
\end{proposition}
\begin{proof}
Consider the \( \mathbb{Q} \)-divisor \( -\sigma_{i, \nu}(\omega) F_\pi + \textrm{Div}(\overline{\omega}_\nu) \) on \( \overbar{X} \). Let us compute the intersection number of this divisor with \( E_i \) in two different ways. First, notice that the intersection number \( (-\sigma_{i, \nu}(\omega) F_\pi + \textrm{Div}(\overline{\omega}_\nu)) \cdot E_i \) equals
\[ \sum_{j = 1}^r \varepsilon_{j, \nu}(\omega) + \sum_{k = 1}^s \delta_{k, \nu}(\omega) - E_i^2. \]
On the other hand, recall that the ideal sheaf of \( E_i \) in \( \overbar{X} \) is \( \mathcal{O}_{\overbar{X}}(-E_i) \) and that \( \overline{\omega}_\nu \) is a section of \( \Omega_{\overbar{X}}^{2} \otimes \mathcal{O}_{\overbar{X}}(-\nu E_i) \). Hence, \( \textrm{Div}(\overline{\omega}_\nu) \cdot E_i = (K_\pi - \nu E_i) \cdot E_i = -2 - E_i^2 - \nu E_i^2 \), by the adjunction formula for surfaces, see for instance \cite[\S V.1]{hartshorne}.
\end{proof}

\begin{corollary} \label{cor:epsilons}
Assume \( \omega \in \Gamma(X, \Omega_{X}^2) \) is such that the divisors in \( \textnormal{Supp}(\textnormal{Div}(\pi^* \omega)) \) and in \( \textnormal{Supp}(F_\pi) \) intersecting \( E_i \) are the same, then
\[ \sum_{j = 1}^r \varepsilon_{j, \nu}(\omega) \geq -2. \]
\end{corollary}
\begin{proof}
This follows from the assumption and the fact that the line bundle \( \mathcal{O}_{\overbar{X}}(-\nu E_i) \otimes \mathcal{O}_{E_i} \) on \( E_i \) has a degree equal to \( -\nu E_i^2 \).
\end{proof}

\begin{remark}
Notice that this result is not true without the hypothesis on the support of the divisor of \( \pi^* \omega \). For instance, let \( f = (y^2-x^3)^2-x^5y \) and \( \omega = (y^2+x^3)\dd x \wedge \dd y \). On \( E_3 \), the first rupture divisor of the minimal resolution, one has that \( N_3 = 12, v_3(\omega) = 10 \). Then, for \( \nu = 0 \), \( \sigma_{3, 0}(\omega) = -11/12 \) and \( \varepsilon_{1, 0}(\omega) = -2/3, \varepsilon_{2, 0}(\omega) = -1/2, \varepsilon_{3, 0}(\omega) = -11/6\). In addition, \( R_{3, 0}(\omega) \) has an extra zero of multiplicity \( 1 \) in \( E_3^\circ \) given by the strict transform of \( y^2+x^3 \). Then, \( \varepsilon_{1, 0}(\omega) + \varepsilon_{2, 0}(\omega) + \varepsilon_{3, 0}(\omega) = -3 \leq -2 \).
\end{remark}

\jump

Assuming that \( \omega \) satisfies \cref{cor:epsilons}, if \( E_i \) is a rupture divisor, i.e. \( \chi(E_i^\circ) = 2 - r < 0 \), and we suppose that the candidate \( \sigma_{i, \nu}(\omega) - 1\) is such that \( \varepsilon_{j, \nu}(\omega) \not\in \mathbb{Z} \), then \cref{prop:deligne-mostow} holds and the multivalued form \( R_{i, \nu}(\omega) \) defines a non-zero cohomology class on \( H^1(E_i^\circ, L) \). The consequence of this is that, since the pairing between homology and cohomology is non-degenerate, there exists a twisted cycle \( C \in H_1(E_i^\circ, L^\vee) \) such that
\[ \int_{C} R_{i, \nu}(\omega) \neq 0. \]

Following \cite{loeser88}, let \( p : F \longrightarrow E_i^\circ \) be the finite cover associated with the local system \( L \) on \( E_i^\circ \). This finite cover is characterized by the fact that \( p_* \underline{\mathbb{C}}_{F} = L \). By definition, the twisted cycle \( C \in H_1(E_i^\circ, L^\vee) \) is identified with a cycle in \( H_1(F, \mathbb{C}) \). Recall now the morphism \( \rho : \widetilde{X} \longrightarrow \overbar{X} \) from \cref{sec:resolution-sing}. The restriction of \( \rho \) to \( \widetilde{E}_i \) is a ramified covering of degree \( N_i \), the multiplicity of the divisor \( E_i \), ramified at the points \( E_i \cap D_j \) with monodromies \( \exp{( 2 \pi \imath N_j / N_i)} \). Therefore, since the monodromies of \( F \) are
\[ \exp(2 \pi \imath \varepsilon_{j, \nu}(\omega)) = \exp\left(-2 \pi \imath (k_i + 1 + \nu) \frac{N_j}{N_i}\right) = \exp\left(2 \pi \imath \frac{N_j}{N_i}\right)^{-(k_i + 1 + \nu)}, \]
the restriction \( \rho_i \) of \( \rho \) to \( \widetilde{E}_i^\circ \) factorizes as
\[ \rho_i : \widetilde{E}_i^\circ \xrightarrow{\ \ q \ \ } F_0 \xrightarrow{\ p|_{F_0} \ } E_i^\circ, \]
where \( F_0 \) is a given connected component of \( F \). Now, since \( q \) is also a finite covering, there exists an integer \( m \) and a cycle \( \tilde{\gamma} \) in \( H_1(\widetilde{E}_i^\circ, \mathbb{C}) \) such that \( q_* \tilde{\gamma} = m C \). Finally, since \( \tilde{f} \) is a locally trivial fibration in a neighborhood of \( \widetilde{E}^\circ_i \), using tubular neighborhoods, we can extend \( \tilde{\gamma} \) to a family of locally constant cycles \( \tilde{\gamma}(\tilde{t}) \) in \( H_1(\widetilde{X}_t, \mathbb{C}) \) with \( \tilde{t} \in \widetilde{T}' \) such that they vanish to \( \tilde{\gamma}(0) := \tilde{\gamma} \), see \cite[\S 4.3]{varchenko82}.

\jump

Setting \( \gamma(\tilde{t}^e) := \rho_* \tilde{\gamma}(\tilde{t}) \) for every point \( t \in T' \) in the base we have obtained, under some assumptions on the candidate exponent \( \sigma_{i, \nu}(\omega) - 1 \) and the exceptional divisor \( E_i \), a family of locally constant cycles in \( H_1(\overbar{X}_t, \mathbb{C}) \) such that they satisfy \cref{eq:construction-loeser}. The precise result is stated in the proposition below.

\begin{definition} \label{def:non-resonant}
A candidate \( b \)-exponent \( \sigma_{i, \nu}(\omega) \) will be called \emph{non-resonant} if the numbers \( \varepsilon_{j, \nu}(\omega) \), defined in \cref{eq:definition-epsilons}, belong to \( \mathbb{Q}\setminus \mathbb{Z} \).
\end{definition}

Notice that the non-resonance condition is independent of the form \( \omega \in \Gamma(X, \Omega^2_X) \) chosen to define \( \sigma_{i, \nu}(\omega) \). That is, if \( \omega' \in \Gamma(X, \Omega^2_X) \) is another differential form such that \( \sigma_{i, \nu}(\omega') = \sigma_{i, \nu}(\omega) \) and \( \sigma_{i, \nu}(\omega) \) is non-resonant, then \( \sigma_{i, \nu}(\omega') \) is also non-resonant.

\begin{proposition} \label{prop:non-zero-section}
Let \( \omega \in \Gamma(X, \Omega^2_X) \) be a differential form satisfying \cref{cor:epsilons} and such that the \( \nu \)--th piece \( \overline{\omega}_{\nu} \) of \( \omega \) for the rupture divisor \( E_i \) is non-zero. Assume also that \( \sigma_{i, \nu}(\omega) \) is non-resonant. Then, the multivalued differential form \( R_{i, \nu}(\omega) \) on \( E_i \) defines a non-zero locally constant geometric section \( A_{\sigma_{i, \nu}-1, 0}^\omega(t) \) of the vector bundle \( H^1 \).
\end{proposition}
\begin{proof}
We have seen that under the hypothesis of the proposition there exists a vanishing cycle \( \gamma(t) \) such that \( \langle A_{\sigma_{i, \nu}-1, 0}^\omega(t), \gamma(t) \rangle \) is non-zero. For any other vanishing cycle \( \gamma'(t) \) of the bundle \( H_1 \), if its limit cycle \( \gamma' \) defines a cycle \( C' \) of \( H_1(E_i^\circ, L^\vee) \), the pairing \( \langle A_{\sigma_{i, \nu}-1, 0}^\omega(t) , \gamma'(t) \rangle = \langle R_{i, \nu}(\omega), C' \rangle \) is well-defined since \( R_{i, \nu}(\omega) \) defines a cohomology class of \( H^1(E_i^\circ, L) \). It may happen that \( \gamma' \) does not define a cycle in \( E_i^\circ \). However, in this case, \( \gamma' \) defines a locally finite homology class \( C' \) of \( H_1^{lf}(E_i^\circ, L^\vee) \). Now, since \( \sigma_{i, \nu}(\omega) \) is non-resonant, one has that \( H_1^{lf}(E_i^\circ, L^\vee) \cong H_1(E_i^\circ, L^\vee) \), see \cite[Prop. 2.6.1]{deligne-mostow86}, and \( C' \) can be replaced by a cycle \( C'' \) in \( H_1(E_i^\circ, L^\vee) \) for which \(\langle A_{\sigma_{i, \nu}-1, 0}^\omega(t), \gamma'(t) \rangle = \langle R_{i, \nu}(\omega), C'' \rangle \) is well-defined.
\end{proof}

If in \cref{prop:non-zero-section} one sets \( \nu = 0 \) and takes, for instance, \( \omega = \dd x \wedge \dd y \) one obtains the results of Lichtin \cite[Prop.~1]{lichtin89}, for the case of irreducible plane curves, and Loeser \cite[Prop.~III.3.2]{loeser88}, for general plane curves. In this situation the first piece \( \overline{\omega}_0 \) for any exceptional divisor \( E_i \) is always non-zero. For irreducible plane curves, Lichtin \cite[Prop. 2.12]{lichtin85} proves that the exponents \( \sigma_{i, 0}(\dd x \wedge \dd y) \) are always non-resonant. This result is related to the fact that for irreducible plane curve singularities the monodromy endomorphism is of finite order, see \cite[Thm. 3.3.1]{trang72, acampo73bis}.

\subsection{Dual locally constant geometric sections} \label{sec:dual-geometric}

In this section, we will continue to work on a fixed exceptional divisor \( E_i \) of the resolution. We will show that, under some assumptions on the combinatorics of the exceptional divisors, the locally constant geometric sections \( A^\omega_{\sigma_{i, \nu}-1, 0} \) from \cref{prop:non-zero-section} are dual with respect to the exceptional divisor \( E_i \). This concept of duality with respect to \( E_i \) will be clear at the end of the section, but it essentially means that the locally constant section \( A^\omega_{\sigma_{i, \nu}-1, 0} \) will be dual to some eigenvector of the monodromy with respect to a basis of cycles vanishing to \( E_i^\circ \). This is a first step towards constructing a basis of locally constant geometric sections of \( H^1 \) dual to a certain basis of \( H_1 \) that is needed for the semicontinuity of the \( b \)-exponents presented in \cref{sec:semicontinuity}.

\jump

Consider the germ of a holomorphic function
\[ \bar{f} := \pi^* f : (\overbar{X}, E) \longrightarrow (\mathbb{C}, 0), \]
with \( E = \pi^{-1}(\boldsymbol{0}) \), that arises after resolution of singularities. Abusing the notation, we will also denote by \( \bar{f} \) the representative \( \bar{f} : \overbar{X} \longrightarrow T \) of the germ. The restriction of \( \bar{f} \) to \( \overbar{X} \setminus \bar{f}^{-1}(0) \) is a locally trivial fibration and the fiber \( \overbar{X}_t := \overbar{X} \cap \bar{f}^{-1}(t), t \in T' \) is analytically isomorphic to the Milnor fiber \( X_t \), since \( \pi \) is an isomorphism outside the singular locus of \( f \). The fundamental group of \( T' \) induces a geometric monodromy map \( \bar{h} \) on \( \overbar{X}_t\) which coincides with the monodromy map \( h \) on \( X_t \) after conjugation by \( \restr{\pi}{\overbar{X}_t} \).

\jump

Let \( T_{j, r} \subset \overbar{X}, 0 < r \ll 1, \) be small enough tubular neighborhoods of the divisors \( D_j \) in \( \textnormal{Supp}(F_\pi) \) having non-empty intersection with \( E_i \). Define
\[ \overbar{X}_{i, t} := \big( \overbar{X}_{i} \setminus \cup_j T_{j, r} \big) \cap \bar{f}^{-1}(t), \qquad t \in T', \]
the subset of the Milnor fiber \( \overbar{X}_t \) over \( E_i \setminus E_i \cap (\cup_j T_{j, r}) \)\footnote{It is enough to take \( r \) small enough so that the cycles constructed in \cref{sec:plane-curves} lie on \( \overbar{X}_{i, t}\).}. Since on \( \overbar{X}_{i} \setminus \cup_j T_{j, \epsilon} \) the map \( \bar{f} \) can be written locally as \( \bar{f} = x_0^{N_i} \), the set \( \overbar{X}_{i, t} \) is locally defined by \( x_0^{N_i} = t \). The limit to the exceptional fiber \( \overbar{X}_0 \) is then given by \( x_0^{N_i} = 0 \) where \( x_0 \) is a local equation of \( E_i \). Notice that \( \overbar{X}_{i, t} \) is an \( N_i \)--fold covering of \( E_i \setminus E_i \cap (\cup_j T_{j, r}) \). Let \( \phi_t : \overbar{X}_{i, t} \longrightarrow E_i^\circ \) be the projection map. In this situation, the geometric monodromy \( \bar{h} \) restricts to \( \overbar{X}_{i, t} \subset \overbar{X}_t \). Indeed, for any point of \( \overbar{X}_{i, t} \) there are local coordinates \( x_0, \dots, x_n \) such that the point satisfies \( x_0^{N_i} = t \). Then, \( \bar{h} \) is just given by \( (x_0, x_1, \dots, x_n) \mapsto (e^{2 \pi \imath/N_i}x_0, x_1, \dots, x_n) \), and the image is again a point of \( \overbar{X}_{i, t} \). In addition, the restricted monodromy coincides with the deck transformation of the cover \( \phi_t \).

\jump

Denoting by \( \bar{h}' \) the restriction of the monodromy map \( \bar{h} \) to \( \overbar{X}_{i, t} \) and by \( j : \overbar{X}_{i,t} \hookrightarrow \overbar{X}_t \) the open inclusion, we obtain the following commutative diagram,

\begin{equation*}
\begin{tikzcd}
\overbar{X}_{i, t} \arrow[hook', swap]{d}{j} \arrow{r}{\bar{h}'} & \overbar{X}_{i, t} \arrow[hook', swap]{d}{j} \arrow{r}{\phi_t} & E_i^\circ \arrow[hook']{d} \\
\overbar{X}_{t} \arrow{r}{\bar{h}} & \overbar{X}_t & \overbar{X}_0.
\end{tikzcd}
\end{equation*}
This commutative diagram then induces a commutative diagram on homology. Namely,
\begin{equation} \label{eq:diagram-homology}
\begin{tikzcd}
H_n(\overbar{X}_{i, t}, \mathbb{C}) \arrow[swap]{d}{j_*} \arrow{r}{\bar{h}'_*} & H_n(\overbar{X}_{i, t}, \mathbb{C}) \arrow[swap]{d}{j_*} \arrow{r}{(\phi_t)_*} & H_n(E_i^\circ, \mathbb{C}) \arrow{d} \\
H_n(\overbar{X}_{t}, \mathbb{C}) \arrow{r}{\bar{h}_*} & H_n(\overbar{X}_t, \mathbb{C}) & H_n(\overbar{X}_0, \mathbb{C}).
\end{tikzcd}
\end{equation}
The vanishing cycles in \( j_*H_n(\overbar{X}_{i, t}, \mathbb{C}) \) are precisely those vanishing cycles from \( H_n(\overbar{X}_{t}, \mathbb{C}) \) that vanish to a cycle in \( E^\circ_i \).

\jump
Fix a locally constant geometric section \( A^\omega_{\sigma_{i, \nu}-1, 0} \) from \cref{prop:non-zero-section}, we construct next the cycle which will be dual to \( A^\omega_{\sigma_{i, \nu}-1, 0} \). Take the non-zero cycle \( \gamma(t) \in H_1(\overbar{X}_t, \mathbb{C}) \), given by \cref{prop:non-zero-section}, such that \( \langle A^\omega_{\sigma_{i, \nu}-1, 0}(t), \gamma(t) \rangle \neq 0 \). For \( \lambda := \exp{(-2 \pi \imath \sigma_{i, \nu}(\omega))} \), consider the projection \( \gamma_{\lambda}(t) \) of \( \gamma(t) \) to the subspace of \( H_1(\overbar{X}_t, \mathbb{C}) \) where \( (\bar{h}_* - \lambda \Id) \) acts nilpotently. Notice that Equations \eqref{eq:integral-expansion-simple} and \eqref{eq:integral-expansion-total}, imply that
\[ \langle A^\omega_{\sigma_{i, \nu}-1, 0}(t), \gamma_{\lambda}(t) \rangle = \langle A^\omega_{\sigma_{i, \nu}-1, 0}(t), \gamma(t) \rangle \neq 0, \]
which implies \( \gamma_{\lambda}(t) \neq 0 \). By construction, the vanishing cycle \( \gamma(t) \) belongs to the subspace \( j_*H_n(\overbar{X}_{i, t}, \mathbb{C}) \subset H_n(\overbar{X}_{t}, \mathbb{C}) \). We will show in \cref{prop:cycle} below that \( \gamma_{\lambda}(t) \) is dual to \( A^\omega_{\sigma_{i, \nu}-1, 0}(t) \) with respect to any basis of the monodromy \( \bar{h} \) restricted to the subspace \( j_* H_1(\overbar{X}_{i,t}, \mathbb{C}) \).

\jump

First, recall that the characteristic polynomial \( \Delta(t) \) of the monodromy endomorphism \( h_* : H_n(X_t, \mathbb{C}) \longrightarrow H_n(X_t, \mathbb{C}) \) is determined by A'Campo \cite{acampo75} in terms of a resolution of the singularity of \( f : (\mathbb{C}^{n+1}, \boldsymbol{0}) \longrightarrow(\mathbb{C}, 0) \). Precisely, A'Campo proved the following formula for \( \Delta(t) \):
\begin{equation} \label{eq:acampo}
  \Delta(t) = \left[ \frac{1}{t-1} \prod_{i = 1}^r \big(t^{N_i} - 1\big)^{\chi(E_i^\circ)} \right]^{(-1)^{n}}.
\end{equation}
For the two-dimensional case, i.e. \( n = 1 \), one has a similar result to \cref{eq:acampo} for the action of the monodromy \( \bar{h}'_* \) on \( H_n(\overbar{X}_{i, t}, \mathbb{C}) \).
\begin{proposition} \label{prop:char-poly}
Let \( E_i \) be an exceptional divisor with multiplicity \( N_i \). Then, the characteristic polynomial \( \Delta_i(t) \) of the monodromy endomorphism \( \bar{h}'_* : H_1(\overbar{X}_{i, t}, \mathbb{C}) \longrightarrow H_1(\overbar{X}_{i,t}, \mathbb{C}) \) is equal to
\[ \Delta_{i}(t) = (t^{N_i} - 1)^{- \chi(E_i^\circ)}(t^{c_i}-1), \]
where \( c_i = \gcd(N_j \ |\ D_j \cap E_i \neq \emptyset, D_j \in \textnormal{Supp}(F_\pi)) \).
\end{proposition}

The proof of this proposition, which uses the same ideas as the proof of \cref{eq:acampo} given in \cite{acampo75}, will be the content of the next section.

\begin{proposition} \label{prop:cycle}
Under the assumptions of \cref{prop:non-zero-section}, let \( \sigma_{i, \nu}(\omega) \) be non-resonant associated with a rupture divisor \( E_i \) such that \( \chi(E_i^\circ) = -1 \). Then, \( \gamma_{\lambda}(t) \) is an eigenvector of the monodromy \( \bar{h}^* \) with eigenvalue \( \lambda = \exp(-2 \pi \imath \sigma_{i, \nu}(\omega)) \) belonging to the subspace \( j_* H_1(\overbar{X}_{i, t}, \mathbb{C}) \). Furthermore, it is dual to \( A^\omega_{\sigma_{i, \nu}-1, 0}(t) \) with respect to any basis of the monodromy restricted to the subspace \( j_* H_1(\overbar{X}_{i,t}, \mathbb{C}) \).
\end{proposition}
\begin{proof}
Since the diagram \eqref{eq:diagram-homology} commutes, the monodromy endomorphism \( \bar{h}^* \) restricts to \( j_* H_1(\overbar{X}_{i, t}, \mathbb{C}) \). For this reason, the projection \( \gamma_\lambda(t) \) of \( \gamma(t) \) to the subspace where \( \bar{h}^* - \lambda \Id \) acts nilpotently is also an element of \( j_* H_1(\overbar{X}_{i, t}, \mathbb{C}) \).

\jump

If we assume that \( \chi(E_i^\circ) = -1 \), then by \cref{prop:char-poly} there is only one eigenvalue equal to \( \lambda \) in \( \Delta_i(t) \). Indeed, notice that since \( \sigma_{i, \nu}(\omega) \) is non-resonant, the eigenvalue \( \lambda \) cannot be a root of the factor \( t^{c_i} - 1 \) of \( \Delta_i(t) \). The fact that \( \bar{h}'_* \) has only one eigenvalue equal to \( \lambda \), implies that \( A^\omega_{\sigma_{i, \nu}-1, 0}(t) \) is dual to \( \gamma_{\lambda}(t) \) with respect to any basis of the monodromy \( \bar{h}^* \) restricted to \( j_* H_1(\overbar{X}_{i, t}, \mathbb{C}) \).
\end{proof}
In \cref{sec:generic-bexponents}, for irreducible plane curves, we will merge the cycles and the locally constant geometric sections with respect to each rupture exceptional divisor \( E_i \) to form a basis of locally constant geometric sections of the bundle \( H^1 \) dual to a basis of \( H_1 \). From this result, one may recover the classical result that the monodromy of irreducible plane curves is of finite order, see \cite{trang72, acampo73bis}.

\subsection{Partial characteristic polynomial of the monodromy} \label{sec:char-poly}
In the celebrated work \cite{acampo75} of A'Campo, the characteristic polynomial of the monodromy \( h_* : H_{n}(X_t, \mathbb{C}) \longrightarrow H_{n}(X_t, \mathbb{C}) \) is computed in terms of a resolution of singularities. A'Campo constructs a homotopic model \( F_t \) of the Milnor fiber \( \overbar{X}_t \), so that there is a continuous retraction \( c_t : F_t \longrightarrow X_0 \) from the general to the exceptional fiber, which is compatible with the geometric monodromy. Then, he uses Leray's spectral sequence associated with this map to compute the Lefschetz numbers of the monodromy which in turn determine the zeta function of the monodromy.

\jump

We will show next how the same argument works to prove \cref{prop:char-poly}. Here, we will use the map \( \phi_t : \overbar{X}_{i, t} \longrightarrow E_i^\circ \) of the unramified covering. Over the sets \( E_i^\circ \), the homotopic model \( F_t \) and the fiber \( \overbar{X}_t \) are homeomorphic. To simplify the notation, in the sequel, we will denote also by \( \bar{h} \) the restriction of the monodromy \( \bar{h} \) to the sets \( \overbar{X}_{i, t} \subset \overbar{X}_{t} \). First, recall the zeta function of the monodromy
\[ Z_{\bar{h}}(t) := \prod_{q \geq 0} \det\big(\Id - t{\bar{h}^*}; H^q(\overbar{X}_{i, t}, \mathbb{C})\big)^{{(-1)}^{q+1}}. \]
If we denote the Lefschetz numbers of the monodromy in the following way
\[ \Lambda(\bar{h}^k) = \sum_{q \geq 0} (-1)^q \tr\big((\bar{h}^*)^k; H^q(\overbar{X}_{i, t}, \mathbb{C})\big), \]
the zeta function \( Z_{\bar{h}}(t) \) depends on the integers \( \Lambda(\bar{h}^k) \) through the following well-known inversion formula. Let \( s_1, s_2, \dots \) be the integers defined by \( \Lambda(\bar{h}^k) = \sum_{i | k} s_i \) for \( k \geq 0 \), then
\[ Z_{\bar{h}}(t) = \prod_{i \geq 0} (1 - t^i)^{-s_i/i}. \]
Now, one can use Leray's spectral sequence of the map \( \phi_t \) to compute the Lefschetz numbers \( \Lambda(\bar{h}^k) \). Define the \emph{sheaf of cycles vanishing to the divisor} \( E_i \) as \( \Psi_i^q := R^q{\phi_t}_* \underline{\mathbb{C}}_{\overbar{X}_{i,t}} \). If one replaces \( \phi_t \) by the retraction \( c_t \), one gets the usual sheaf of vanishing cycles, see for instance \cite[Exp. XIII]{sga7-2}. The second page of Leray's spectral sequence of the map \( \phi_t \) is then equal to
\[ E_2^{p, q} = H^p(E_i^\circ, \Psi_i^q), \]
and the spectral sequence converges to \( E_\infty^{p+q} = H^{p+q}(\overbar{X}_{i, t}, \mathbb{C}) \). Since \( \phi_t \) is compatible
with the geometric monodromy, the monodromy endomorphism \( \bar{h}^* \) induces a monodromy action \( T \)  on the sheaf of vanishing cycles of the divisor \( E_i \), namely
\[ T(U) := \big(\bar{h}\big|_{\phi^{-1}_t(U)}\big)^* : H^*(\phi^{-1}_t(U), \mathbb{C}) \longrightarrow H^*(\phi^{-1}_t(U), \mathbb{C}), \]
with the actions \( T_2^{p, q} \) on the terms \( E_2^{p,q} \) converging to \( T^{p+q}_\infty = \bar{h}^* : H^{p+q}(\overbar{X}_{i,t}, \mathbb{C}) \longrightarrow H^{p+q}(\overbar{X}_{i,t}, \mathbb{C}) \). Therefore, one has that
\[ \Lambda(\bar{h}^k) = \Lambda(T^k, E^{\bullet \bullet}_\infty) = \cdots = \Lambda(T^k, E_3^{\bullet \bullet}) = \Lambda(T^k, E_2^{\bullet \bullet}), \]
see for instance \cite[Thm. 4.3.14]{spanier81}. It is then enough to compute \( \Lambda(T^k, E_2^{\bullet \bullet}) \). Again, one argues similarly to \cite{acampo75}. Since \( \phi_t \) is a locally trivial fibration, the sheaves of vanishing cycles of the divisor \( E_i \), \( \Psi^\bullet_i \), are complex local systems. Notice that when one considers the whole exceptional fiber \( \overbar{X}_0 \), the usual sheaves of vanishing cycles are not local systems but constructible sheaves. Hence,
\[ \Lambda(T^k, E^{\bullet \bullet}_2) = \sum_{p,q \geq 0} \tr\big(T^k; H^p(E_i^\circ, \Psi_i^q)\big) = \chi(E_i^\circ) \Lambda(T^k, (\Psi_i^\bullet)_s),  \]
with \( s \in E_i^\circ \). The stalk \( (\Psi_i^\bullet)_s \) is identified with the cohomology of the Milnor fibration of the equation of \( E_i \) at \( s \in E_i^\circ \), and this identification is compatible with both monodromies, see \cite{acampo73}. Since locally at \( s \in E_i^\circ \), \( E_i \) is \( x_0^{N_i} = 0 \) for some local coordinate \( x_0 \), counting fixed points gives
\[ \Lambda(T^k, (\Psi_i^\bullet)_s) =  \begin{cases}
  \, 0, & \textrm{if} \quad  N_i \centernot\mid k, \\
  \, N_i, & \textrm{if} \quad N_i \mid k.
  \end{cases}
\]
Finally, the zeta function of the monodromy \( \bar{h}_* : H_n(\overbar{X}_{i, t}, \mathbb{C}) \longrightarrow H_n(\overbar{X}_{i, t}, \mathbb{C}) \) equals to
\[ Z_{\bar{h}}(t) = \big(1-t^{N_i}\big)^{-\chi(E_i^\circ)}. \]

In order to prove \cref{prop:char-poly}, it remains to compute the characteristic polynomial \( \Delta_i(t) \) of the monodromy action from the zeta function.

\begin{proof}[Proof of \cref{prop:char-poly}]
If we now restrict to the case \( n = 1 \), then the only homology groups are \( H_i(\overbar{X}_{i, t}, \mathbb{C}) \) for \( i = 0,1 \). Hence, in terms of the zeta functions the characteristic polynomial \( \Delta_i(t) \) reads as
\[ \Delta_i(t) = t^{b_1} \bigg[ \frac{t^{b_0}-1}{t} Z_{\bar{h}}(1/t) \bigg], \]
where \( b_0 \) and \( b_1 \) are, respectively, the dimensions of \( H_0(\overbar{X}_{i,t}, \mathbb{C}) \) and \( H_1(\overbar{X}_{i,t}, \mathbb{C}) \).

\jump

Let us now compute the dimension of these homology groups. First, one has that \( \chi(\overbar{X}_{i, t}) = N_i \chi(E_i^\circ) \) since \( \overbar{X}_{i, t} \) is a \( N_i \)--fold unramified covering of \( E_i^\circ \). Second, the number of connected components \( c_i \) of a covering equals the index of the fundamental group of the base in the covering group.

\jump

Fix a point \( x \in E_i^\circ \). The fundamental group \( \pi_1(E_i^\circ, x) \) has rank \( r_i - 1 \), where \( r_i \) is the number of missing points from \( E_i^\circ \), and it is generated by loops \( \gamma_1, \dots, \gamma_{r_i} \) around the missing points with the relation \( \gamma_{r_i} \gamma_{r_{i-1}} \cdots \gamma_{1} = 1 \). On the other hand, the covering group is cyclic and is generated by the monodromy action \( \bar{h} \). Hence, the action of a loop \( \gamma_j \) around the intersection of \( E_i \) with \( E_j \) in the covering group is \( \bar{h}^{N_i/ \gcd(N_i,N_j)} \). Therefore, the index of \( \pi_1(E_i^\circ, x) \) in the covering group is
\[ c_i = \gcd(N_j \ |\ E_j \cap E_i \neq \emptyset, E_j \in \textnormal{Supp}(F_\pi)) = b_0, \]
and \( b_1 = c_i - N_i \chi(E_i^\circ) \). Finally,
\[ \Delta_{i}(t) = (t^{N_i} - 1)^{- \chi(E_i^\circ)}(t^{c_i}-1). \]
\end{proof}

\subsection{Plane branches} \label{sec:plane-branches}

Let \( f : (\mathbb{C}^2, \boldsymbol{0}) \longrightarrow (\mathbb{C}, 0) \) be a holomorphic function defining a plane branch with characteristic sequence \( (n, \beta_1, \dots, \beta_g) \). The goal of this section is to prove the existence of a particular one-parameter \( \mu \)-constant deformation of any plane branch \( f \), see \cref{prop:deformation}, that will be key for the proof of Yano's conjecture. To that end, we will review next some classic results about semigroups of plane branches and we will introduce Teissier's monomial curve and its deformations. For further references about these topics, the reader is referred to \cite{casas}, \cite{ctc-wall} and \cite{zariski-moduli}.

\jump

Let us first recall some well-known facts about plane curve singularities. The notion of equisingularity classifies plane curves having the same combinatorial structure in its minimal resolution process \cite{zariski65-1}. This classification coincides with the topological equivalence as germs in the plane \cite{zariski32, brauner}. In addition, by the results in \cite{trang-ramanujam76}, a \( \mu \)-constant family of reduced plane curves is topologically trivial and hence equisingular.

\jump

Let \( v_f : \mathcal{O}_{X, \boldsymbol{0}} \longrightarrow \mathbb{Z}_{+} \) be the curve valuation in the local ring \( \mathcal{O}_{X, \boldsymbol{0}} \) given by the intersection multiplicity with \( f \). Denote by \( \Gamma := \{v_f(g) \in \mathbb{Z} \ |\ g \in \mathcal{O}_{X, \boldsymbol{0}} \setminus (f)\} \subseteq \mathbb{Z}_{+} \) the semigroup of values of this valuation. Since \( f \) is irreducible, the semigroup is finitely generated, that is, \( \Gamma = \langle \overline{\beta}_0, \dots, \overline{\beta}_g \rangle \) with \( \gcd(\overline{\beta}_0, \dots, \overline{\beta}_g) = 1 \). The semigroup of a plane branch is a complete topological invariant, see \cite[Cor. 5.8.5]{casas}. It is well-known that the generators of the semigroup are determined by the characteristic sequence. For instance, with the notations from \cref{sec:yanos-conjecture}, one has that \( \overline{\beta}_i = R'_i \), see for instance \cite[Thm. 3.9]{zariski-moduli}.

\jump

Let us introduce some extra notation associated with a semigroup \( \Gamma = \langle \overline{\beta}_0, \dots, \overline{\beta}_g \rangle\) in addition to the notation from \cref{eq:yano-definitions}. First, let \( e_i := \gcd(\overline{\beta}_0, \dots, \overline{\beta}_i), e_0 := \overline{\beta}_0 \) and set \( n_i := e_{i-1} / e_i \) for \( i = 1, \dots, g \). Notice that the integers \( n_1, \dots, n_g \) are strictly larger than 1 and we have that \( e_{i-1} = n_i n_{i+1} \cdots n_g \). Finally, define the integers \( m_i := \beta_i /e_i \) and \( \overline{m}_i := \overline{\beta}_i /e_i \) which will be useful in the sequel. The conductor \( c \) of a semigroup is the minimum integer such that for all \( \alpha \geq c \), \( \alpha \) belongs to the semigroup. For the semigroup \( \Gamma \) the conductor \( c \) can be computed, see \cite[\S II.3]{zariski-moduli}, by the formula
\[ c = n_g \overline{\beta}_g - \beta_g - n + 1. \]

\jump

All these numerical data associated with the semigroup allows us to describe some numerical data of a resolution of the plane branch \( f \) more easily. In the sequel, we will assume that the rupture divisors of the minimal resolution of \( f \) are labeled \( E_i, i = 1, \dots, g \). Using this convention, if \( v_i \) denotes the divisorial valuation on \( \mathcal{O}_{X, \boldsymbol{0}} \) associated with the \( i \)--th rupture divisor of \( f \), then \( v_i(f) = N_i = n_i \overline{\beta}_i \), for \( i = 1, \dots, g \), see \cite[Thm. 8.5.2]{ctc-wall}. Notice also that these quantities coincide with the values \( R_i \) defined in \cref{sec:yanos-conjecture}. In addition, \( k_i + 1 = m_i + n_1 \cdots n_i \), which coincides with the value \( r_i \) from Equations \eqref{eq:yano-definitions}. Similarly, let \( f_0, f_1, \dots, f_g \in \mathcal{O}_{X, \boldsymbol{0}} \) be irreducible plane curves such that \( v_f(f_i) = \overline{\beta}_i \). It is not hard to see that the semigroup of \( f_i \), \( i \geq 2 \), is
\begin{equation} \label{eq:semigroup-vi}
\Gamma_i = e_{i-1}^{-1} \langle \overline{\beta}_0, \dots, \overline{\beta}_{i-1} \rangle = \langle n_1 n_2 \cdots n_{i}, n_2 \cdots n_{i} \overline{m}_{1}, \dots, n_{i-1} \overline{m}_{i-2}, \overline{m}_{i-1} \rangle
\end{equation}
and that, \( v_j(f_i) = n_{j-1} \cdots n_i \overline{m}_i \), for \( j \geq i \), see \cite[\S 5.8]{casas}. Moreover, the semigroup of the divisorial valuation \( v_i \) coincides with \( \Gamma_i \).

\jump

Before introducing the monomial curve associated with a semigroup \( \Gamma \), first, we need the following property of the semigroups coming from plane branches.

\begin{lemma}[{\cite[Lemma 2.2.1]{teissier-appendix}}] \label{lemma:semigroup}
If \( \Gamma = \langle \overline{\beta}_{0}, \overline{\beta}_{1}, \dots, \overline{\beta}_{g} \rangle \) is the semigroup of a plane branch, then
\begin{equation} \label{eq:lemma-semigroup}
n_i \overline{\beta}_i \in \langle \overline{\beta}_{0}, \overline{\beta}_1, \dots, \overline{\beta}_{i-1} \rangle \quad \textnormal{for} \quad 1 \leq i \leq g.
\end{equation}
\end{lemma}

\cref{lemma:semigroup} together with the fact that \( n_i \overline{\beta}_{i} < \overline{\beta}_{i+1} \), which follows, for instance, from Equations \eqref{eq:yano-definitions}, completely characterizes the semigroups coming from plane branches, see \cite[\S I.3.2]{teissier-appendix}. Let us now introduce the monomial curve \( C^\Gamma \).

\jump

Let \( \Gamma = \langle \overline{\beta}_0, \overline{\beta}_1, \dots, \overline{\beta}_g \rangle \subseteq \mathbb{Z}_{+}\) be a semigroup such that \( \mathbb{Z}_{+} \setminus \Gamma \) is finite, that is, \( \gcd(\overline{\beta}_0, \dots, \overline{\beta}_g) = 1 \). We do not need to assume yet that the semigroup comes from a plane branch. Following \cite[\S I.1]{teissier-appendix}, let \( (C^\Gamma, \boldsymbol{0}) \subset (Y, \boldsymbol{0}) \) be the curve defined via the parameterization
\[ C^\Gamma : u_i = t^{\overline{\beta}_i}, \quad 0 \leq i \leq g, \]
where \( Y := \mathbb{C}^{g+1} \). The germ \( (C^\Gamma, \boldsymbol{0} ) \) is irreducible since \( \gcd(\overline{\beta}_0, \dots, \overline{\beta}_g) = 1 \) and its local ring \( \mathcal{O}_{C^\Gamma, \boldsymbol{0}} \) equals \( \mathbb{C}\big\{ t^{\overline{\beta}_0}, \dots, t^{\overline{\beta}_g} \big\} \), which has a natural structure of graded subalgebra of \( \mathbb{C}\{t\} \). The first important property of the monomial curve \( C^\Gamma \) is the following.
\begin{theorem}[{\cite[Thm. 1]{teissier-appendix}}] \label{thm:teissier-appendix1}
Every branch \( (C, \boldsymbol{0}) \) with semigroup \( \Gamma \) is isomorphic to the generic fiber of a one-parameter complex analytic deformation of \( (C^\Gamma, \boldsymbol{0}) \).
\end{theorem}
Moreover, with some extra assumptions on the semigroup \( \Gamma \), it is possible to obtain more structure on the singularity of \( (C^\Gamma, \boldsymbol{0}) \) and even explicit equations.
\begin{proposition}[{\cite[Prop. 2.2]{teissier-appendix}}] \label{prop:teissier-appendix2}
If \( \Gamma \) satisfies \eqref{eq:lemma-semigroup}, then the branch \( (C^\Gamma, \boldsymbol{0}) \subset (Y, \boldsymbol{0}) \) is a quasi-homogeneous complete intersection with equations
\begin{equation} \label{eq:monomial-complete-intersection}
h_i := u_i^{n_i} - u_0^{l^{(i)}_0} u_1^{l^{(i)}_1} \cdots u_{i-1}^{l^{(i)}_{i-1}} = 0, \quad 1 \leq i \leq g,
\end{equation}
and weights \( \overline{\beta}_0, \overline{\beta}_1, \dots, \overline{\beta}_g \), where
\[ n_i \overline{\beta}_i = \overline{\beta}_{0} l_0^{(i)} + \cdots + \overline{\beta}_{i-1} l_{i-1}^{(i)} \in \langle \overline{\beta}_{0} , \dots, \overline{\beta}_{i-1} \rangle. \]
\end{proposition}

In the sequel, we will always assume that the semigroup \( \Gamma \) fulfills \eqref{eq:lemma-semigroup}. Since \( C^\Gamma \) is an isolated singularity one has the existence of the miniversal deformation of \( C^\Gamma \), see \cite[Add.]{teissier-appendix}. After \cref{thm:teissier-appendix1}, every curve with semigroup \( \Gamma \), is analytically isomorphic to one of the fibers of the miniversal deformation of \( C^\Gamma \). Moreover, Teissier \cite[\S I.2]{teissier-appendix} proves the existence of the miniversal semigroup constant deformation of \( C^\Gamma \).

\begin{theorem}[{\cite[Thm. 3]{teissier-appendix}}] \label{thm:teissier}
There exists a miniversal semigroup constant deformation
\[ G : (Y_\Gamma, C^\Gamma) \longrightarrow (\mathbb{C}^{\tau_{-}}, \boldsymbol{0}) \]
of \( (C^\Gamma, \boldsymbol{0}) \) with the property that, for any branch \( (C, \boldsymbol{0}) \) with semigroup \( \Gamma \), there exists \( \boldsymbol{v}_C \in \mathbb{C}^{\tau_{-}} \) such that \( (G^{-1}(\boldsymbol{v}_C), \boldsymbol{0}) \) is analytically isomorphic to \( (C, \boldsymbol{0}) \).
\end{theorem}

\jump

Since \( C^\Gamma \) is a complete intersection and we have equations for \( C^\Gamma \), the miniversal semigroup constant deformation can be made explicit, see \cite[\S I.2]{teissier-appendix}. Indeed, consider the Tjurina module of the complete intersection \( (C^\Gamma, \boldsymbol{0}) \),
\[ T^1_{C^\Gamma, \boldsymbol{0}} = \mathcal{O}_{Y, \boldsymbol{0}}^g\Big/\left(J\boldsymbol{h} \cdot \mathcal{O}_{Y, \boldsymbol{0}}^{g+1} + \langle h_1, \dots, h_g \rangle \cdot \mathcal{O}_{Y, \boldsymbol{0}}^g\right), \]
where \( J \boldsymbol{h} \cdot \mathcal{O}_{Y, \boldsymbol{0}}^{g+1} \) is the submodule of \( \mathcal{O}_{Y, \boldsymbol{0}}^g \) generated by the columns of the Jacobian matrix of the morphism \( \boldsymbol{h} = (h_1, \dots, h_g) \). Since \( (C^\Gamma, \boldsymbol{0}) \) is an isolated singularity, \( T^1_{C^\Gamma, \boldsymbol{0}} \) is a finite-dimensional \( \mathbb{C} \)--vector space of dimension usually denoted by \( \tau \), the Tjurina number of the singularity. Moreover, since \( (C^\Gamma, \boldsymbol{0}) \) is Gorenstein, \(\tau = 2 \cdot \# (\mathbb{Z}_{>0} \setminus \Gamma)\), see \cite[Prop. 2.7]{teissier-appendix}.

\jump

Let \( \boldsymbol{\phi}_1, \dots, \boldsymbol{\phi}_\tau \) be a basis of \( T^1_{C^\Gamma, \boldsymbol{0}} \). It is easy to see that we can take representatives for the vectors \( \boldsymbol{\phi}_r \) in \( \mathcal{O}_{Y}^g \) having only one non-zero monomial entry \( \phi_{r, i} \). Since \( (C^\Gamma, \boldsymbol{0}) \) is quasi-homogeneous, we can endow \( T^1_{C^\Gamma, \boldsymbol{0}} \) with a structure of graded module in such a way that, using only a certain subset of the basis \( \boldsymbol{\phi}_1, \dots, \boldsymbol{\phi}_{\tau_{-}} \), \( \tau_{-} < \tau \), \( Y_\Gamma \) is defined from \cref{eq:monomial-complete-intersection} as the subvariety of \( \mathbb{C}^{g+1} \times \mathbb{C}^{\tau_{-}} \), with coordinates \( (u_i, v_j) \), given by the equations
\begin{equation} \label{eq:monomial-curve}
  H_i := h_i + \sum_{r = 1}^{\tau_{-}} v_r \phi_{r, i}(u_0, \dots, u_g) = 0,  \quad 1 \leq i \leq g,
\end{equation}
with the weighted degree of \(\phi_{r,i} \) strictly bigger than \( n_i \overline{\beta}_i\), see \cite[Thm. 2.10]{teissier-appendix} for the exact construction.

\jump

Take now the classes of the vectors \( u_2 \boldsymbol{e}_1 := (u_2, 0, \dots, 0), \dots, u_g \boldsymbol{e}_{g-1} := (0, \dots, u_g, 0) \) on \( T^1_{C^\Gamma, \boldsymbol{0}} \) which are linear independent over \( \mathbb{C} \). If we assume now that \( \Gamma \) is the semigroup of a plane branch, then \( n_i \overline{\beta}_i < \overline{\beta}_{i+1}, 1 \leq i < g \), and the vectors \( u_2 \boldsymbol{e}_1, \dots, u_g \boldsymbol{e}_{g-1} \) are part of the miniversal semigroup constant deformation of \( (C^\Gamma, \boldsymbol{0}) \). Notice also that, a fiber of the semigroup constant deformation will be a plane branch, i.e. will have embedding dimension equal to two, if and only if, the coefficients of the vectors \( u_i \boldsymbol{e}_{i-1}, i = 1, \dots, g \) in the deformation from \cref{eq:monomial-curve} are all different from zero.

\begin{proposition} \label{prop:deformation}
Let \( f : (\mathbb{C}^2, \boldsymbol{0}) \longrightarrow (\mathbb{C}, 0) \) be a plane branch. Let \( E_i \) be a rupture divisor of the minimal resolution of \( f \) with divisorial valuation \( v_i \). Then, for any \( v > n_i \overline{\beta}_i \) there exists a one-parameter \( \mu \)-constant deformation of \( f \) of the form \( f + tg_t \) such that \( v_i(g_t) = v \), for all values of the parameter \( t \).
\end{proposition}
\begin{proof}
By the above discussion, there exists a fiber of the miniversal semigroup constant deformation of the monomial curve \( C^\Gamma \) analytically equivalent to the germ of curve \( (C, \boldsymbol{0}) \) defined by \( f \). Precisely, with the notations above,
\begin{equation} \label{eq:prop-deformation}
C : h_j + \lambda_i u_{j+1} + l_j = 0 \quad \textnormal{for} \quad 0 \leq j \leq g,
\end{equation}
with \( \lambda_{i} \neq 0 \) and where we set \( u_{g+1} := 0 \) for convenience. Here, the elements \( l_j \) are linear combinations of the non-zero components of the elements \( \boldsymbol{\phi}_1, \dots, \boldsymbol{\phi}_{\tau_{-}} \) that are different from the elements \( u_2 \boldsymbol{e}_1, \dots, u_g \boldsymbol{e}_{g-1} \), see \cref{eq:monomial-curve}. Let now \( v' := n_i \overline{m}_i + v - n_i \overline{\beta}_i \), then \( v' \) belongs to the semigroup
\[ \Gamma_{i+1} = \langle n_1 n_2 \cdots n_{i}, n_2 \cdots n_{i} \overline{m}_{1}, \dots, n_{i} \overline{m}_{i-1}, \overline{m}_{i} \rangle\]
of the divisorial valuation \( v_i \), see \cref{eq:semigroup-vi}. Indeed, notice that \( v' \) is strictly larger than the conductor of \( \Gamma_{i+1} \). Therefore, there exists \( (\alpha_0, \dots, \alpha_i) \in \mathbb{Z}_{+}^{i+1} \) such that
\[ v' = \alpha_0 n_1 n_2 \cdots n_{i} + \alpha_1 n_2 \cdots n_{i} \overline{m}_{1} + \cdots + \alpha_g \overline{m}_{i}. \]
Consider the one-parameter deformation of \( C \) given by deforming the \( i \)--th equation of \eqref{eq:prop-deformation} in the following way
\[ h_i + \lambda_i u_{i+1} + l_i + t u_0^{\alpha_0} u_1^{\alpha_1} \cdots u_i^{\alpha_i} = 0. \]
First, we must check that this deformation is semigroup constant, which is equivalent to seeing that \( \deg (u_0^{\alpha_0} \cdots u_g^{\alpha_g}) > n_i \overline{\beta}_i \). Indeed,
\[ \begin{split}
\deg(u_0^{\alpha_0} u_1^{\alpha_1} \cdots u_g^{\alpha_g}) & = \alpha_0 \overline{\beta}_0 + \alpha_1 \overline{\beta}_1 + \cdots + \alpha_g \overline{\beta}_g \\
& = e_i(\alpha_0 n_1 n_2 \cdots n_i + \alpha_1 n_2 \cdots n_i \overline{m}_1 + \cdots + \alpha_g \overline{m}_i) \\
& = n_i \overline{\beta}_i + e_i (v - n_i \overline{\beta}_i) > n_i \overline{\beta}_i.
\end{split}
\]
Now, eliminating the variables \( u_2, u_3, \dots, u_g \) successively from \cref{eq:prop-deformation} one gets a deformation of \( f \) of the form \( f + tg_0 + t^2g_1 + \cdots \), where the dots mean higher-order terms on \( t \). For plane branches, a family with constant semigroup is \( \mu \)-constant, so this proves the first part of the lemma.

\jump

For the second part, it is not hard to see that when the first \( u_2, u_3, \dots, u_{i-1} \) variables are eliminated, \( u_i = 0 \) defines a plane branch \( f_i \) such that \( v_f(f_i) = \overline{\beta}_i = \deg u_i \). Finally, one can see that
\[ g_0 = f_0^{\alpha_0} \cdots f_i^{\alpha_i} f_{i+1}^{n_{i+1}-1} \cdots f_{g}^{n_g-1} + \cdots \]
with precisely
\begin{equation} \label{eq:987}
v_i(f_0^{\alpha_0} \cdots f_i^{\alpha_i} f_{i+1}^{n_{i+1}-1} \cdots f_{g}^{n_g-1}) = v.
\end{equation}
In order to check \cref{eq:987}, recall that \( v_i(f_j) = n_{j-1} \cdots n_i \overline{m}_i \), for \( j \geq i \), and then
\[
\begin{split}
v_i(f_0^{\alpha_0} \cdots f_i^{\alpha_i} f_{i+1}^{n_{i+1}-1} \cdots f_g^{n_g -1}) & = v' + (n_{i+1}-1) v_i(f_{i+1}) + \cdots + (n_g - 1) v_i(f_g) \\
    & = v' + (n_i - 1)n_i \overline{m}_i + \cdots + (n_g - 1) n_{g-1} \cdots n_i \overline{m}_i \\
    & = n_g n_{g-1} \cdots n_i \overline{m}_i + v - n_i \overline{\beta}_i = v.
\end{split}
\]
A similar computation shows that any other term in \( g_0 \), or in higher-order powers of \( t \), has a value strictly larger than \( v \) for the valuation \( v_i \) of the divisor \( E_i \).
\end{proof}

\subsection{Generic \texorpdfstring{\( b \)}{b}-exponents} \label{sec:generic-bexponents}

Let \( f : (\mathbb{C}^2, \boldsymbol{0}) \longrightarrow (\mathbb{C}, 0) \) be a germ of an holomorphic function defining an irreducible plane curve with semigroup \( \Gamma = \langle \overline{\beta}_0, \dots, \overline{\beta}_g \rangle \). Given \( E_i \) a rupture divisor of the minimal resolution of \( f \), take \( \sigma_{i, \nu}(\omega) \) a non-resonant candidate \( b \)-exponent associated with \( E_i \), {see \cref{eq:def-sigma} and \cref{def:non-resonant}.

\begin{lemma} \label{lemma:non-resonance}
A candidate \( b \)-exponent \( \sigma_{i, \nu}(\omega) \) is non-resonant, if and only if, \( \overline{\beta}_i \sigma_{i, \nu}(\omega) \not\in \mathbb{Z} \) and \( e_{i-1} \sigma_{i, \nu}(\omega) \not\in \mathbb{Z} \).
\end{lemma}
\begin{proof}
The candidate \( \sigma_{i, \nu}(\omega) \) is non-resonant if \( \varepsilon_{j, \nu}(\omega) \not\in \mathbb{Z} \), for all \( D_j \cap E_i \neq \emptyset, D_j \in \textrm{Supp}(F_\pi) \). By the definition of \( \varepsilon_{j, \nu}(\omega) \), this is the same as \( N_j \sigma_{i, \nu}(\omega) \not\in \mathbb{Z} \). Since for plane branches there are only three divisors \( D_1, D_2, D_3 \) crossing \( E_i \) in the support of \( F_\pi \), by \cref{prop:epsilons}, and since the \( \delta_{k, \nu}(\omega) \) are integers, it is enough to check the non-resonance condition for two of the crossing divisors. Therefore, assume \( D_1, D_2 \) are the divisors preceding \( E_i \) in the minimal resolution. Hence, \( N_1, N_2 < N_i = n_i \overline{\beta}_i \) and \( N_j \sigma_{i, \nu} \not\in \mathbb{Z} \) is equivalent to \( \gcd(N_i, N_j) \sigma_{i, \nu} \not\in \mathbb{Z} \), \( j = 1, 2 \). However, after a possible reordering, \( \gcd(N_i, N_1) = \overline{\beta}_i \) and \( \gcd(N_i, N_2) = e_{i-1} \), see \cite[Prop. 8.5.3]{ctc-wall}.
\end{proof}

\newcommand*{\bigbigcup}{\mathop{\scalebox{1.5}{\ensuremath{\bigcup}}}}%

For the rest of the section, \( \omega \in \Gamma(X, \Omega^2_X) \) will be a fixed top differential form such that \( \omega = g \dd x \wedge \dd y \) with \( g(\boldsymbol{0}) \neq 0 \). For simplicity, one can take \( \omega = \dd x \wedge \dd y \). Following \cref{sec:plane-branches}, we can write the candidates associated with such \( \omega \) for each rupture divisor \( E_i \) in terms of the semigroup \( \Gamma \) in the following way,
\[ \sigma_{i, \nu}(\omega) = \frac{m_i + n_1 \cdots n_i + \nu}{n_i \overline{\beta}_i}, \qquad \nu \in \mathbb{Z}_{+}. \]

Notice now that the set of candidates from Yano's conjecture, see \cref{eq:yano-generating-series}, are exactly
\begin{equation} \label{eq:equation-candidates}
\bigbigcup_{i = 1}^g \bigg\{ \sigma_{i, \nu}(\omega) = \frac{m_i + n_1 \cdots n_i + \nu}{n_i \overline{\beta}_i} \ \bigg|\ 0 \leq \nu < n_i \overline{\beta}_i,\  \overline{\beta}_i \sigma_{i, \nu}(\omega), e_{i-1} \sigma_{i, \nu}(\omega) \not\in \mathbb{Z} \bigg\}.
\end{equation}
To see the equality between the exponents of \cref{eq:yano-generating-series} and the set in \eqref{eq:equation-candidates}, it is enough to notice that \( R'_i = \overline{\beta}_i\) and \(r'_i = \lceil (m_i + n_1 \cdots n_i)/n_i \rceil \). Hence, \( R_i = N_{i} = n_i \overline{\beta}_i = n_i R'_i \) and \( r_i = k_{i} + 1 = m_i + n_1 \cdots n_i = n_i r'_i \).

\jump

If we consider A'Campo formula in the case of plane branches, it is easy to see that there are exactly \( \mu \) elements in the sets from \eqref{eq:equation-candidates}, counted with possible multiplicities. Therefore, \( \lambda = \exp{(2 \pi \imath \sigma_{i, \nu}(\omega))} \) with \( i = 1, \dots, g \) and  \( 0 \leq \nu < n_i \overline{\beta}_i \) is the set of all the eigenvalues of the monodromy of a plane branch.

\begin{proposition} \label{prop:non-zero-section2}
Let \( \lambda = \exp(-2 \pi \imath \sigma_{i, \nu}(\omega)), 0 \leq \nu < n_i \overline{\beta}_i, \) be an eigenvalue of the monodromy. For any \( f_y : (\mathbb{C}^2, \boldsymbol{0}) \longrightarrow (\mathbb{C}, 0), y \in I_\delta, \) \( \mu \)-constant deformation of \( f \), there exists a differential form \( \eta_y \in \Gamma(X, \Omega^2_X) \) such that \( A^{\eta_y}_{\sigma_{i, 0}-1, 0}(t, y) \) is non-zero for all fibers of the deformation and \( \exp(-2 \pi \imath \sigma_{i,0}(\eta_y)) = \exp{(-2 \pi \imath \sigma_{i, \nu}(\omega)}) \).
\end{proposition}
\begin{proof}
First, recall that a \( \mu \)-constant deformation is topologically trivial, see \cite{trang-ramanujam76}, and hence equisingular. Recall also that the semigroup \( \Gamma_i \) of the divisorial valuation \( v_i \) associated with the rupture divisor \( E_i \) with candidate exponent \( \sigma_{i, \nu}(\omega) \) is finitely generated, see \cref{eq:semigroup-vi}. Take \( k \gg 0 \), such that \( \nu' = \nu + k N_i \) is larger than the conductor of the semigroup \( \Gamma_i \). Now, let \( h_y \in \Gamma(X, \mathcal{O}_{X}) \) with \( v_i(h_y) = \nu' \) and define \( \eta_y = h_y \dd x \wedge \dd y \). Notice that \( h_y \) can always be chosen such that \( \eta_y \) satisfies \cref{cor:epsilons}. Then, because
\begin{equation*}
  \sigma_{i, 0}(\eta_y) = \frac{k_i + \nu' + 1}{N_i} = \frac{k_i + \nu + kN_i + 1}{N_i} = \sigma_{i, \nu}(\omega) + k,
\end{equation*}
the eigenvalues of the monodromy are the same, and \( \sigma_{i, 0}(\eta_y) \) is non-resonant since \( \sigma_{i, \nu}(\omega) \) is non-resonant by \cref{lemma:non-resonance}. Finally, the first piece \( \overline{\eta}_{y,0} \) of \( \eta_y \) associated with \( E_i \) is non-zero. Therefore, by \cref{prop:non-zero-section}, the locally constant section \( A^{\eta_y}_{\sigma_{i, 0} -1, 0}(t, y) \) defined by \( R_{i, 0}(\eta_y) \) is non-zero.
\end{proof}

We can now use the previous proposition together with \cref{prop:cycle} to construct dual bases of locally constant sections of the bundles \( H_1 \) and \( H^1 \) for the fibers of a one-parameter \( \mu \)-constant deformation of a plane branch.

\begin{theorem}[Semicontinuity] \label{thm:semicontinuity}
If \( f: (\mathbb{C}^2, \boldsymbol{0}) \longrightarrow (\mathbb{C}, 0) \) is a plane branch, the \( b \)-exponents of a one-parameter \( \mu \)-constant deformation of \( f \) depend upper-semicontinuously on the parameter.
\end{theorem}
\begin{proof}
For a fixed \( 1 \leq i \leq g \), let \( \lambda := \exp(-2 \pi \imath \sigma_{i, \nu}(\omega)), 0 \leq \nu < n_i \overline{\beta}_i \) be an eigenvalue of the monodromy with \( \sigma_{i, \nu}(\omega) \) from \eqref{eq:equation-candidates}. After \cref{prop:non-zero-section2}, there is a differential form \( \eta_y \) with \( \lambda = \exp{(-2 \pi \imath \sigma_{i, 0}(\eta_y))} \) such that there exists a non-zero locally constant section \( A^{\eta_y}_{\sigma_{i, 0}-1, 0}(t, y) \) for all values of the parameter \( y \). Since for plane branches \( \chi(E_i^\circ) = -1 \), we can apply \cref{prop:cycle} to this section, and for \( t \in T' \), we obtain the existence of \( \gamma_{\lambda}(t, y) \) a representative 1-cycle of an eigenvector of the monodromy of the subspace \( j_* H_1(\overbar{X}_{i, t}, \mathbb{C}) \) with eigenvalue \( \lambda \).

\jump

The set of homology classes of all such cycles for \( \sigma_{i, \nu}(\omega), 0 \leq \nu < n_i \overline{\beta}_i, i = 1, \dots, g, \) in \eqref{eq:equation-candidates} gives a basis of eigenvectors \( \gamma_{\lambda}(t, y) \) of the monodromy endomorphism which are dual to the corresponding \( A^{\eta_y}_{\sigma_{i, 0}-1, 0}(t, y) \), for all fibers of the \( \mu \)-constant deformation. Indeed, since we have exactly \( \mu \) cycles and all these subspaces \( j_* H_1(\overbar{X}_{j, t}, \mathbb{C}) \) are direct summands in \( H_1(\overbar{X}_t, \mathbb{C}) \), one has that
\[ H_1(\overbar{X}_t, \mathbb{C}) = j_* H_1(\overbar{X}_{1, t}, \mathbb{C}) \oplus j_* H_1(\overbar{X}_{2, t}, \mathbb{C}) \oplus \cdots \oplus j_* H_1(\overbar{X} _{g, t}, \mathbb{C}). \]
After \cref{prop:cycle}, any such \( A^{\eta_y}_{\sigma_{i, 0}-1, 0} \) is dual to the eigenvectors forming a basis of \( j_* H_1(\overbar{X}_{i, t}, \mathbb{C}) \). Finally, since any other vanishing cycle not in \( j_* H_1(\overbar{X}_{i, t}, \mathbb{C}) \) must vanish to a different rupture divisor, i.e. other than \( E_i \), we obtain the desired duality.

\jump

For this precise basis of eigenvectors of the monodromy we constructed, we have shown the existence of dual locally constant geometric sections for all fibers of the deformation. That is, using the notations from \cref{sec:semicontinuity}, for all \( \gamma_{\lambda}(t, y) \) in the basis, \( \dim_{\mathbb{C}} H^1_{\gamma_\lambda}(y) = 1 \), for all values of the parameter \( y \). Therefore, we can apply \cref{prop:semicontinuity-bexponent} and all the \( b \)-exponents of any one-parameter \( \mu \)-constant deformation depend upper-semicontinuously on the deformation parameter.
\end{proof}

Finally, Yano's conjecture will follow from \cref{thm:semicontinuity} and the following proposition. For any \( \nu \in \mathbb{Z}_{+} \), we can show that, generically in a one-parameter \( \mu \)-constant deformation of \( f \), the piece \( \overline{\omega}_\nu \) of degree \( \nu \) of \( \omega \) associated with a rupture divisor \( E_i \) is non-zero, and hence \( R_{i, \nu}(\omega) \) is non-zero.

\begin{proposition} \label{prop:non-zero-piece}
For any \( \sigma_{i, \nu}(\omega), \nu \in \mathbb{Z}_{+} \) non-resonant, there exists \( f_y : (\mathbb{C}^2, \boldsymbol{0}) \longrightarrow (\mathbb{C}, 0), y \in I_\delta, \) a one-parameter \( \mu \)-constant deformation of \( f \) such the locally constant section \( A^\omega_{\sigma_{i, \nu}-1, 0}(t, y) \) is non-zero for generic fibers of the deformation.
\end{proposition}
\begin{proof}
Assume that \( \nu > 0 \). Let \( f_y = f + yg_y \) be the one-parameter \( \mu \)-constant deformation of \( f \) from \cref{prop:deformation} with \( v_i(g_y) = N_i + \nu = n_i \overline{\beta}_i + \nu \). Recall that, since the deformation is \( \mu \)-constant, all the fibers are equisingular. Thus, locally in \( \overbar{X}_i \), let \(x\) be a local defining equation for \(E_i^\circ\) and \( z \) the other coordinate. Then, we can write
\begin{equation}
f_y = x^{N_i} + y x^{N_i + \nu} u_y(x, z),
\end{equation}
since \( v_i(g_y) = N_i + \nu \) and where \( u_y(0, z) \) is not identically zero. Then, this is equal to \( x^{N_i}(1 + yx^{\nu} u_y) \) and the curves \(f_y\) can be written locally as \( \bar{x}^{N_i} \) for a new coordinate \( \bar{x} \).

\jump

Focusing on the differential form, we have that \( \overline{\omega} = x^{k_i} z^{b_i} v(x, z) \dd x \wedge \dd z, b_i \geq 0 \), with \( v(x, z) \) a unit. Now, expand \( v(x, z) \) in series and apply the change of coordinates \( x = \bar{x}(1-y\bar{x}^\nu \bar{u}_y(\bar{x}, z)) \) which comes from inverting \( \bar{x} = x(1+y x^\nu u_y(x, z))^{1/N_i} \) with respect to \( x \), and \( \bar{u}_y(0, z) \) is not identically zero. That is,
\begin{equation} \label{eq:eq666}
\overline{\omega} = \sum_{\alpha, \beta \geq 0} a_{\alpha, \beta} \bar{x}^{k_i + \alpha}(1-y \bar{x}^\nu \bar{u}_y)^{k_i + \alpha} z^{b_i + \beta} \dd \left(\bar{x}(1-y \bar{x}^\nu \bar{u}_y)\right) \wedge \dd z,
\end{equation}
notice that the differential form depends now on the deformation parameter \( y \). We have to check that, generically on \( y \), the \( \nu \)--th piece \( \overline{\omega}_{\nu} \) of \( \overline{\omega} \) is non-zero. In order to study \( \overline{\omega}_{\nu} \), we look at the terms of \( \overline{\omega} \) with degree \( k_i + \nu \) in \( \bar{x} \). Since \( \dd x \wedge \dd z = [1 - (\nu + 1) y \bar{x}^\nu \bar{u}_y - y x^{\nu + 1} \partial \bar{u}_y/\partial x] \dd \bar{x} \wedge \dd z \), the only relevant terms from \cref{eq:eq666} are,
\begin{equation*}
a_{\nu, \beta} \bar{x}^{k_i + \nu} z^{\beta}, \quad - a_{0, \beta} k_i y \bar{x}^{k_i + \nu} \bar{u}_y z^{\beta}, \quad - a_{0,\beta} (\nu + 1) y \bar{x}^{k_i + \nu} \bar{u}_y z^{\beta}.
\end{equation*}
Since \( a_{0, 0} \neq 0 \), \( \overline{\omega}_{\nu} \) is non-zero for \( y \neq 0, 0 < |y| \ll 1 \), as we wanted to show. Since \( \sigma_{i, \nu}(\omega) \) is non-resonant, and since the pull-back of \( \omega \) has exceptional support, we can apply \cref{prop:non-zero-section}. This implies that \( R_{i, \nu}(\omega) \), and hence \( A^\omega_{\sigma_{i, \nu}-1, 0}(t, y) \), is a non-zero cohomology class as required.

\jump

Finally, notice that for the case \( \nu = 0 \), it is enough to consider the trivial deformation since \( \overline{\omega}_{0} \) is always different from zero.
\end{proof}

\begin{theorem}[Yano's conjecture] \label{thm:yano}
Let \( f : (\mathbb{C}^2, \boldsymbol{0}) \longrightarrow (\mathbb{C}, 0) \) be a germ of a holomorphic function defining an irreducible plane curve with semigroup \( \Gamma = \langle \overline{\beta}_0, \overline{\beta}_1, \dots, \overline{\beta}_g \rangle \). Then, for generic curves in some \( \mu \)-constant deformation of \( f \), the \( b \)-exponents are
\begin{equation} \label{eq:generic-bexponents}
\bigbigcup_{i = 1}^g \bigg\{ \sigma_{i, \nu} = \frac{m_i + n_1 \cdots n_i + \nu}{n_i \overline{\beta}_i} \ \bigg|\ 0 \leq \nu < n_i \overline{\beta}_i,\  \overline{\beta}_i \sigma_{i, \nu}, e_{i-1} \sigma_{i, \nu} \not\in \mathbb{Z} \bigg\}.
\end{equation}
\end{theorem}
\begin{proof}
Let \( \sigma_{i, \nu}(\omega) \) be a candidate \( b \)-exponent from the set \eqref{eq:generic-bexponents}. By \cref{lemma:non-resonance}, \( \sigma_{i, \nu}(\omega) \) is non-resonant and, as a consequence of \cref{prop:non-zero-piece}, we have the existence, generically in a \( \mu \)-constant deformation of \( f \), of non-zero locally constant geometric section \( A^\omega_{\sigma_{i, \nu}-1, 0} \) given by the exponent \( \sigma_{i, \nu}(\omega) - 1 \) and associated with the rupture divisor \( E_i \).

\jump

After \cref{lemma:b-exponents} and \cref{thm:varchenko}, checking that \( A^\omega_{\sigma_{i, \nu}-1, 0} \) is not a section of \( S_{\sigma_{i, \nu}(\omega)-2} \) will prove that \( \sigma_{i, \nu}(\omega) \) is a \( b \)-exponent. But this fact follows from the condition \( 0 \leq \nu < n_i \overline{\beta}_i \) from \eqref{eq:generic-bexponents}. Indeed, no locally constant geometric section \( A^{\omega}_{\alpha -1, 0} \) associated to the exceptional divisor \( E_i \) can have an exponent \( \alpha - 1\) smaller than \( \sigma_{i, 0}(\omega) - 1\). Thus, \( A^\omega_{\sigma_{i, \nu}-1, 0} \) cannot be a section of \( S_{\sigma_{i, \nu}(\omega)-2} \) because \( \sigma_{i, \nu}(\omega) - 1 < \sigma_{i, 0}(\omega) \).

\jump

Finally, we can use the upper-semicontinuity result from \cref{thm:semicontinuity} to apply this argument to all the candidate \( b \)-exponents from \eqref{eq:generic-bexponents}. In this case, since \( \sigma_{i, \nu}(\omega) -1, 0 \leq \nu < n_i \overline{\beta}_i, \) is smaller than \( \sigma_{i, 0}(\omega) \), the upper-semicontinuity implies that when a single candidate \( b \)-exponent has been set generically, further deformations do not change that \( b \)-exponent. This way, we obtain a \( \mu \)-constant deformation of the original curve \( f \) such that all the candidates from \eqref{eq:generic-bexponents} are the \( b \)-exponents of generic fibers of this \( \mu \)-constant deformation.
\end{proof}

\bibliography{short,main}
\bibliographystyle{amsplain}

\end{document}